\documentclass[11pt]{article}

\usepackage[dvipsnames]{xcolor}
\usepackage{amsmath, amsthm, amssymb, amsfonts, listings, hyperref, multicol, bm, shuffle, xcolor, enumerate, dsfont, tikz, dsfont, tabu, HeisenbergMACROS,graphicx,mathtools,array,todonotes}
\usepackage{mathabx}

\usepackage{tikz}
\usetikzlibrary{decorations.pathreplacing}
\usepackage{tikz}

\usepackage{fancyheadings}
\theoremstyle{definition}
\newtheorem{theorem}{Theorem}[section]
\newtheorem{corollary}[theorem]{Corollary}

\newtheorem{lemma}[theorem]{Lemma}
\newtheorem{remark}[theorem]{Remark}
\newtheorem{proposition}[theorem]{Proposition}
\newtheorem{definition}[theorem]{Definition}
\newtheorem{example}[theorem]{Example}

\usetikzlibrary{decorations.pathreplacing,shapes}
\usepackage{tikz}

\newcommand{\arrowlines}{%
\begin{tikzpicture}
  \draw[thick] (0,0) -- (-.1,-.1);
  \draw[thick] (0,0) -- (.1,-.1);
\end{tikzpicture}%
}

\newcommand{\ckpicture}{
\begin{tikzpicture}
\draw[thick] (2,2) circle (.5cm);
\node at (1.5,2) {\arrowlines};
\draw[fill=black] (1.62,2.33) circle (.08cm);
\node at (1.3,2.7) {$k$};
\end{tikzpicture}}

\newcommand{\cktildepicture}{
\begin{tikzpicture}
\draw[thick] (2,2) circle (.5cm);
\node[rotate = 180] at (1.5,2) {\arrowlines};
\draw[fill=black] (1.62,2.33) circle (.08cm);
\node at (1.3,2.7) {$k$};
\end{tikzpicture}}

\definecolor{lightblue}{RGB}{240,240,255}
\definecolor{lightred}{RGB}{255,240,240}

\newcommand{\omitt}[1]{}
\newcommand{\tonytodo}{\todo[inline,color=pink!20]}

\title{Khovanov's Heisenberg category, moments in free probability, and shifted symmetric functions}

\date{}

\author{Henry Kvinge, Anthony M. Licata\thanks{AML was supported by a Discovery Project grant from the Australian Research Council.}, and Stuart Mitchell}

\begin{document}
\maketitle
\begin{abstract}
We establish an isomorphism between the center $\Hcenter$ of the Heisenberg category defined by Khovanov in \cite{Kho14} and the algebra $\ShiftSym{}$ of shifted symmetric functions defined by Okounkov-Olshanski in \cite{OO97}.  We give a graphical description of the shifted power and Schur bases of $\ShiftSym{}$ as elements of $\Hcenter$, and describe the curl generators of $\Hcenter$ in the language of shifted symmetric functions.  This latter description makes use of the transition and co-transition measures of Kerov \cite{Ker93} and the noncommutative probability spaces of Biane \cite{B98}.
  \end{abstract}

\tableofcontents

%


\section{Introduction}
In \cite{Kho14}, Khovanov introduces a graphical calculus of oriented planar diagrams and uses it to define a linear monoidal category $\Heisencat$, which he proposes as a categorification of the Heisenberg algebra. 
We denote by $\Hcenter$ the endormophism algebra of the monoidal unit in $\Heisencat$.  The commutative algebra 
$\Hcenter$ is, by definition, the algebra of closed oriented planar diagrams modulo the relations of the Khovanov graphical calculus.  In his study of morphism spaces of $\Heisencat$, Khovanov introduces two sets of generators for $\Hcenter$: the clockwise curls $\{c_k\}_{k \geq 0}$ and the counterclockwise curls $\{\ctilde{k}\}_{k\geq 2}$.  He then establishes algebra isomorphisms
$$
	\Hcenter \cong \MB{C}[c_0,c_1,c_2, \dots] \cong \MB{C}[\ctilde{2},\ctilde{3},\ctilde{4}, \dots],
$$
and describes a recursion for expressing the clockwise and counterclockwise curls in terms of each other.
He then relates $\Heisencat$ to representation theory by defining a sequence of monoidal functors $\htobimod{k}$ from $\Heisencat$ to bimodule categories for symmetric groups.  A consequence of the existence of these functors is the existence of surjective algebra homomorphisms,
$$
	f_n^{\Heisencat} : \Hcenter \longrightarrow  Z(\MB{C}[\Sy{n}]),
$$
from $\Hcenter$ to the center of the group algebra of each symmetric group.  Based in part on this, Khovanov suggests that there should be a close connection between $\Hcenter$ and the asymptotic representation theory of symmetric groups.  Furthermore, one might hope that $\Hcenter$ in fact gives a diagrammatic description of some algebra of pre-existing combinatorial interest.  

The main goal of the current paper is to make precise the connection between $\Hcenter$ and both the asymptotic representation theory of symmetric groups and algebraic combinatorics.  We do this by establishing an isomorphism between 
$$
	\primaryiso : \Hcenter \longrightarrow \ShiftSym{},
$$
where $\ShiftSym{}$ is the {\emph{shifted symmetric functions}} of Okounkov-Olshanski \cite{OO97}.  (See Theorem \ref{thm-main-1}.)  The algebra of shifted symmetric functions $\ShiftSym{}$ is a deformation of the algebra of symmetric functions.  As is the case for $\Hcenter$, there are surjective algebra homomorphisms
$$
	f_n^{\ShiftSym{}}: \ShiftSym{} \longrightarrow Z(\MB{C}[\Sy{n}]),
$$
to the center of the group algebra of each symmetric group. The isomorphism $\primaryiso : \Hcenter \longrightarrow \ShiftSym{}$ is canonical, in that it intertwines the homomorphisms $ f_n^{\Heisencat} $ and $f_n^{\ShiftSym{}}$.

The isomorphism $\primaryiso:\Hcenter \longrightarrow \ShiftSym{}$ allows us to give a graphical description of several important bases of $\ShiftSym{}$.  For example, the shifted power sum denoted $p_\lambda^\#$ in \cite{OO97} appears in $\Hcenter$ as the closure of a permutation of cycle type $\lambda$.  The shifted Schur function $s_\lambda^*$ appears as the closure of a Young symmetrizer of type $\lambda$.  (See Theorem \ref{thm:schur}).

In the other direction, it is also reasonable to ask for a description of the image of Khovanov's curl generators $c_k$ and $\ctilde{k}$ as elements of $\ShiftSym{}$.  It turns out that the right language for such a description is that of noncommutative  probability theory.  In \cite{Ker93}, Kerov introduces, for each partition $\lambda$, a pair of finitely supported probability measures on $\MB{R}$; these probability measures are known as the \emph{ transition} and {\emph{co-transition} measures, or sometimes as growth and decay.  In work of Biane \cite{B98}, these probability measures appear as the compactly-supported measures associated to self-adjoint operators on a noncommutative probability space, and as a result they are basic objects of interest at the intersection of representation theory and noncommutative probability theory.  In particular, the \emph {moments} and \emph{Boolean cumulants} of the transition and co-transition measures may be regarded as elements of $\ShiftSym{}$.  In Theorem \ref{thm-moment-images}, we show that the isomorphism $\primaryiso$ takes Khovanov's curl generators $c_k$ and $\ctilde{k}$ to scalar multiples of the $k$th moments of Kerov's transition and co-transition measures.  In fact, the close relationship between the transition and co-transition measures themselves yields two independent descriptions of the image of the curl generator $c_k$: it is equal to a scalar multiple of both the $k$th moment of the co-transition measure and the $(k+2)$th Boolean cumulant of the transition measure.  The observation that the Boolean cumulants of the transition measure are equal to the moments of the co-transition measure seems to be new, and is closely connected to the adjointness of induction and restriction functors between representation categories of symmetric groups.  A dictionary between several of the bases of $\Hcenter$ and $\ShiftSym{}$ is given in Table \ref{i-class-table} below.  

\begin{table}
\begin{center}
{\tabulinesep=1.2mm
 \begin{tabular}{| m{2.4cm} |  m{5cm} |}
 \hline
$ \ShiftSym{}$ & diagram in $\Hcenter$   \\
 \hline \hline
  $\shiftpwr{\lambda}$ & \begin{tikzpicture}[scale = .75]

\node at (-2.5,0) {};
\node at (0,2) {};
\draw[thick] (0,0) circle (1.85 cm);
\draw[thick] (0,0) circle (1.6 cm);
\draw[thick] (0,0) circle (1 cm);

\draw[fill=white,thick] (-2,-.3) rectangle (-.5,.3);

\node at (-1,.5) {$\cdot$};
\node at (-1.15,.55) {$\cdot$};
\node at (-1.3,.6) {$\cdot$};

\node at (-1.3,0) {\small{$\lambda$}};

\node[rotate = 180] at (1,0) {\arrowlines};
\node[rotate = 180] at (1.6,0) {\arrowlines};
\node[rotate = 180] at (1.85,0) {\arrowlines};

\node at (-3.6,0) {};

\end{tikzpicture}  \\
\hline
 $\shiftschur{\lambda}$ & \begin{tikzpicture}[scale = .75]

\node at (-2.5,0) {};
\node at (0,2) {};
\draw[thick] (0,0) circle (1.85 cm);
\draw[thick] (0,0) circle (1.6 cm);
\draw[thick] (0,0) circle (1 cm);
\node at (0,-1.7) {};

\draw[fill=white,thick] (-2,-.3) rectangle (-.5,.3);

\node at (-1,.5) {$\cdot$};
\node at (-1.15,.55) {$\cdot$};
\node at (-1.3,.6) {$\cdot$};

\node at (-3,-.3) {$\dim \simplerep{\lambda}$};
\node at (-3,.3) {$1$};

\node at (-1.3,0) {\scriptsize{\text{$\idempotent{\lambda}$}}};

\draw[thick] (-3.8,0) -- (-2.3,0);

\node[rotate = 180] at (1,0) {\arrowlines};
\node[rotate = 180] at (1.6,0) {\arrowlines};
\node[rotate = 180] at (1.85,0) {\arrowlines};

\end{tikzpicture}
   \\
    \hline
 $\homogenshift{k}$ & \begin{tikzpicture}[scale = .75]

\node at (-2.5,0) {};
\node at (0,2) {};
\draw[thick] (0,0) circle (1.85 cm);
\draw[thick] (0,0) circle (1.6 cm);
\draw[thick] (0,0) circle (1 cm);

\draw[fill=white,thick] (-2,-.3) rectangle (-.5,.3);


\node at (-1,.5) {$\cdot$};
\node at (-1.15,.55) {$\cdot$};
\node at (-1.3,.6) {$\cdot$};

\node at (-1.3,0) {\scriptsize{\text{$\idempotent{(k)}$}}};

\node[rotate = 180] at (1,0) {\arrowlines};
\node[rotate = 180] at (1.6,0) {\arrowlines};
\node[rotate = 180] at (1.85,0) {\arrowlines};

\node at (-3.6,0) {};

\end{tikzpicture}   \\
 \hline
 $\elementaryshift{k}$ & \begin{tikzpicture}[scale = .75]

\node at (-2.5,0) {};
\node at (0,2) {};
\draw[thick] (0,0) circle (1.85 cm);
\draw[thick] (0,0) circle (1.6 cm);
\draw[thick] (0,0) circle (1 cm);


\draw[fill=white,thick] (-2,-.3) rectangle (-.5,.3);

\node at (-1,.5) {$\cdot$};
\node at (-1.15,.55) {$\cdot$};
\node at (-1.3,.6) {$\cdot$};

\node at (-1.3,0) {\scriptsize{\text{$\idempotent{(1^k)}$}}};

\node[rotate = 180] at (1,0) {\arrowlines};
\node[rotate = 180] at (1.6,0) {\arrowlines};
\node[rotate = 180] at (1.85,0) {\arrowlines};

\node at (-3.6,0) {};

\end{tikzpicture}  \\
 \hline
 $\moment{k}$ & \quad\quad\quad\quad\quad\cktildepicture  \\
  \hline
 $\Boolean{k+2} = \shiftpwr{1}\comoment{k}$ & \quad\quad\quad\quad\quad\ckpicture  \\
\hline
\end{tabular}}
\end{center}
\caption{\label{i-class-table} A dictionary between $\ShiftSym{}$ and diagrams in $\Hcenter$.}
\label{node-classification}
\end{table}

The existence of a relationship between $\Heisencat$ and free probability -- and indeed, much of this paper -- was anticipated by Khovanov in \cite{Kho14}.  The relationship between generators of $\Hcenter$ and the noncommutative probability spaces of \cite{B98} may be seen as a further manifestation of the ``planar structure" of free probability; the many connections between noncommutative probability and other mathematical subjects with planar structure are emphasized in the work of Guionnet, Jones and Shlyakhtenko \cite{GJS10}.  

In addition to the center of $\Heisencat$, another algebra of interest in the study of $\Heisencat$ is its trace (or zeroth Hochschild homology).  The trace of $\Heisencat$ is an infinite-dimensional noncommutative algebra, which may be defined diagrammatically as the algebra of diagrams on an annulus;  the trace acts naturally on $\Hcenter$ by gluing annular diagrams around planar ones.  In \cite{CLLS15}, the trace of $\Heisencat$ is shown to be isomorphic to the $W_{1+\infty}$ algebra of conformal field theory.  An action  of $W_{1+\infty}$ on $\ShiftSym{}$ appears to be well known in the vertex algebra community, and such an action is constructed explicitly in the work of Lascoux-Thibon \cite{LT01}.  Thus the isomorphism $\primaryiso: \Hcenter\longrightarrow \ShiftSym{}$ of Theorem \ref{thm-main-1}, together with the main result of \cite{CLLS15}, gives a purely planar realization -- via Khovanov's graphical calculus -- of Lascoux-Thibon's construction.

\vspace{-3mm}

\subsection{Acknowledgements}

The authors would like to thank Ben Elias, Alexander Ellis, Sara Billey, Eugene Gorsky, Aaron Lauda, Carson Rogers, and Alistair Savage for helpful conversations. We would also like to thank Monica Vazirani for her valuable comments after a careful reading of an earlier draft of this paper. H.K. would like to thank Mikhail Khovanov for his suggestion to look for a relationship between the Heisenberg category and the combinatorics of symmetric functions.

\section{The symmetric group and its normalized character theory} \label{sect-sym-group}


We begin by establishing notation related to partitions and Young diagrams. Let $\partitionsn{n}$ be the set of partitions of $n$ and 
\begin{equation*}
\partitionsn{} := \bigcup_{n \geq 0} \partitionsn{n}. 
\end{equation*}
For this section let $\lambda = (\lambda_1, \lambda_2, \dots, \lambda_r) \in \partitionsn{n}$ and $\mu = (\mu_1, \dots, \mu_t) \in \partitionsn{k}$ with $n \geq k$. We assume that $\lambda_1 \geq \dots \geq \lambda_r > 0$ and $\mu_1 \geq \dots \geq \mu_t > 0$. When $i > r$ (respectively $i > t$) we then understand $\lambda_i = 0$ (resp. $\mu_i = 0$).
We use the following notation throughout:
\begin{itemize}
\omitt{
\item $\partsofpartition{s}{\lambda}$ is the number of parts of $\lambda$ equal to $s$. We can then also write $\lambda = (1^{\partsofpartition{1}{\lambda}}, 2^{\partsofpartition{2}{\lambda}}, \dots)$. }
\item $n = \lambda_1 + \lambda_2 + \dots + \lambda_r =: |\lambda|$. 
\item $\lambda \cup \mu$ is the partition formed from the union of the parts of $\lambda$ and $\mu$. 
\item $\mu \subseteq \lambda$ if $\mu_i \leq \lambda_i$ for all $i \geq 1$. When this is the case, we write $\lambda / \mu$ for the associated skew diagram.
\item $\partembedding{k}{n}: \partitionsn{k} \hookrightarrow \partitionsn{n}$ is the function defined by $\partembedding{k}{n}(\mu) = \mu \cup 1^{n-k} \in \partitionsn{n}$. 
\omitt{
\item $\nounity{\lambda}$ is the partition obtained by removing all parts of size 1 from $\lambda$. Thus $\partsofpartition{1}{\nounity{\lambda}} = 0$ and $\partsofpartition{s}{\nounity{\lambda}} = \partsofpartition{s}{\lambda}$ for $s \geq 2$.
}
\end{itemize}

\begin{example}
If $\mu = (3,2,1,1,1) \in \partitionsn{8}$ then $\partembedding{8}{10}(\mu) = (3,2,1,1,1,1,1) \in \partitionsn{10}$.
\end{example} 

\omitt{
and $\nounity{\mu} = (3,2)$.
\end{example}
}

\omitt{
Given a set of indeterminants $\{x_i \}_{i \geq 0}$ we set
\begin{equation*} 
x^{\mu} := x_0^{\partsofpartition{1}{\mu}} x_1^{\partsofpartition{2}{\mu}} \dots x_{r-1}^{\partsofpartition{r}{\mu}}.
\end{equation*}
}

We freely identify $\mu \in \partitions$ with its corresponding Young diagram, which we draw using Russian notation (see Example \ref{example-interlacing-sequences}). If $\ydcell$ is a cell in the $i$th row and $j$th column of $\mu$ then the {\emph{content}} of $\ydcell$ is defined as
\begin{equation*}
\content{\ydcell} := j-i.
\end{equation*}
We say that a cell $\ydcell \notin \mu$ is $i$-addable with respect to $\mu$ if it has content $i$ and adding it to $\mu$ gives a Young diagram. We say that a cell $\ydcell \in \mu$ is $i$-removable with respect to $\mu$ if it has content $i$ and removing it from $\mu$ gives a Young diagram. We call two sequences $\interlacingx{1}, \dots, \interlacingx{d}$ and $\interlacingy{1}, \dots, \interlacingy{d-1}$ {\emph{interlacing}} when 
\begin{equation*}
\interlacingx{1} < \interlacingy{1} < \interlacingx{2} < \dots < \interlacingx{d-1} < \interlacingy{d-1} < \interlacingx{d}. 
\end{equation*}
The \emph{center} of this pair of sequences is defined as the quantity $(\interlacingx{1} + \dots + \interlacingx{d}) - (\interlacingy{1} + \dots + \interlacingy{d-1})$. Each Young diagram $\mu$ uniquely defines two integer valued interlacing sequences $\interlacingx{1}, \dots, \interlacingx{d}$ and $\interlacingy{1}, \dots, \interlacingy{d-1}$ where:
\begin{itemize}
\item $\interlacingx{1}, \dots, \interlacingx{d}$ is the ordered list of all $\interlacingx{j}$ such that there exists an $\interlacingx{j}$-addable cell with respect to $\mu$. 
\item $\interlacingy{1}, \dots, \interlacingy{d-1}$ is the ordered list of all $\interlacingy{j}$ such that there exists a $\interlacingy{j}$-removable cell with respect to $\mu$.
\end{itemize}
From this description it is clear that $\interlacingx{1}, \dots, \interlacingx{d}$ and $\interlacingy{1}, \dots, \interlacingy{d-1}$ are interlacing. 

\begin{example} \label{example-interlacing-sequences}
Let $\mu = (4,2,1)$. Then $\mu$ yields the interlacing sequences
\begin{equation*}
{\color{red}{-3}} < {\color{red}{-1}} < {\color{red}{1}} < {\color{red}{4}} \quad\quad \text{and} \quad\quad {\color{blue}{-2}} < {\color{blue}{0}} < {\color{blue}{3}}.
\end{equation*}
\begin{center}
\begin{tikzpicture}

\vspace{3mm}

\node at (2.8,2.3) {\color{red}{\scriptsize{$\interlacingx{4}$}}};
\node at (2.1,2.3) {\color{blue}{\scriptsize{$\interlacingy{3}$}}};
\node at (.7,2.3) {\color{red}{\scriptsize{$\interlacingx{3}$}}};
\node at (0,2.3) {\color{blue}{\scriptsize{$\interlacingy{2}$}}};
\node at (-.7,2.3) {\color{red}{\scriptsize{$\interlacingx{2}$}}};
\node at (-1.4,2.3) {\color{blue}{\scriptsize{$\interlacingy{1}$}}};
\node at (-2.1,2.3) {\color{red}{\scriptsize{$\interlacingx{1}$}}};

\draw (-4,-2.15) -- (4,-2.15);

\node at (2.72,-2.3) {\scriptsize{$4$}};
\node at (2.05,-2.3) {\scriptsize{$3$}};
\node at (1.37,-2.3) {\scriptsize{$2$}};
\node at (.7,-2.3) {\scriptsize{$1$}};
\node at (0,-2.3) {\scriptsize{$0$}};
\node at (-.78,-2.3) {\scriptsize{$-1$}};
\node at (-1.48,-2.3) {\scriptsize{$-2$}};
\node at (-2.15,-2.3) {\scriptsize{$-3$}};

\node at (0,0) {\begin{tikzpicture}[rotate = 45, scale = 1.2]

\draw (0,0) rectangle (.8,.8);
\draw (.8,0) rectangle (1.6,.8);
\draw (1.6,0) rectangle (2.4,.8);
\draw (2.4,0) rectangle (3.2,.8);

\draw (0,.8) rectangle (.8,1.6);
\draw (.8,.8) rectangle (1.6,1.6);

\draw (0,1.6) rectangle (.8,2.4);

\draw[->] (0,0) -- (0,5);
\draw[->] (0,0) -- (5,0);

\draw[dotted] (0,0) -- (2.4,2.4);
\draw[dotted] (-.4,.4) -- (2,2.8);
\draw[dotted] (-.8,.8) -- (1.6,3.2);
\draw[dotted] (-1.2,1.2) -- (1.2,3.6);

\draw[dotted] (.4,-.4) -- (2.8,2);
\draw[dotted] (.8,-.8) -- (3.2,1.6);
\draw[dotted] (1.2,-1.2) -- (3.6,1.2);
\draw[dotted] (1.6,-1.6) -- (4,.8);

\end{tikzpicture}};
\end{tikzpicture}
\end{center}
\end{example}

\begin{proposition} \cite{Ker00}
If $\interlacingx{1}, \dots, \interlacingx{d}$ and $\interlacingy{1}, \dots, \interlacingy{d-1}$ are the pair of interlacing sequences associated to a Young diagram then their center is 0. Conversely, any pair of integer valued interlacing sequences with center 0 are associated to a Young diagram.
\end{proposition}

When $\mu \subseteq \lambda$ and $\lambda / \mu = \ydcell$, then we write $\mu \nearrow \lambda$. In other words, $\mu \nearrow \lambda$ whenever we can obtain $\lambda$ from $\mu$ by adding a single cell. If $\interlacingx{1}, \dots, \interlacingx{d}$ and $\interlacingy{1}, \dots, \interlacingy{d-1}$ are the interlacing sequences associated to $\mu$, then we denote by $\mu^{(i)}$ the Young diagram that we get by adding a cell of content $\interlacingx{i}$, so that
\begin{equation*}
\content{\mu^{(i)} / \mu} = \interlacingx{i}.
\end{equation*}
Similarly, we denote by $\mu_{(i)}$ the Young diagram that we get by removing a cell of content $\interlacingy{i}$ from $\mu$, so that \begin{equation*}
\content{\mu / \mu_{(i)}} = \interlacingy{i}.
\end{equation*}
Note that $\mu_{(i)} \nearrow \mu$, while $\mu \nearrow \mu^{(i)}$.

\begin{example}
If $\mu = (4,2,1)$ as in Example \ref{example-interlacing-sequences}, we have
\begin{equation}
\begin{array}{lcl}
\color{red}{\mu^{(1)}} & = & (4,2,1,1) \\
\color{red}{\mu^{(2)}} & = & (4,2,2) \\
\color{red}{\mu^{(3)}} & = & (4,3,1) \\
\color{red}{\mu^{(4)}} & = & (5,2,1) \\
\end{array}
\quad\quad \text{and} \quad\quad
\begin{array}{lcl}
\color{blue}{\mu_{(1)}} & = & (4,2) \\
\color{blue}{\mu_{(2)}} & = & (4,1,1) \\
\color{blue}{\mu_{(3)}} & = & (3,2,1). \\
\end{array}
\end{equation}
\end{example}

Let $\Sy{n}$ be the symmetric group. $\Sy{n}$ is generated by Coxeter generators $s_1, \dots, s_{n-1}$ where $s_i$ is the adjacent transposition $(i, i+1)$. We identify $\MB{C}[\Sy{0}] \cong \MB{C}$. If $g \in \Sy{n}$ has cycle type $\lambda \vdash n$, then we write $\shape{g} := \lambda$. For $k \leq n$, there is an embedding $\Sy{k} \hookrightarrow \Sy{n}$ called the {\emph{standard embedding}} which sends $\Sy{k}$ to the subgroup generated by $s_1, \dots, s_{k-1}$, which stabilizes $\{k+1,\dots,n\}$ pointwise. We extend this embedding by linearity to get an embedding of group algebras which we denote by $\symembedding{k}{n}: \MB{C}[\Sy{k}] \hookrightarrow \MB{C}[\Sy{n}]$. We write $1_k$ for the identity element in $\MB{C}[\Sy{k}]$ so that $\symembedding{k}{n}(1_k) = 1_n$. We write $\longestelement{n}$ for the longest element of $\Sy{n}$.

For $\lambda \vdash n$, let $\simplerep{\lambda}$ be the simple $\MB{C}[\Sy{n}]$-module corresponding to $\lambda$, $\Youngidempotent{\lambda}$ its associated Young idempotent, and $\charrep{\lambda}: \MB{C}[\Sy{n}] \rightarrow \MB{C}$ its associated character. Abusing notation, we write $\charrep{\lambda}(\mu)$ for $\charrep{\lambda}(g)$ when $\shape{g} = \mu$ (this notation is well-defined since $\charrep{\lambda}$ is a class function). The {\emph{normalized character}} $\norcharrep{\lambda}: \bigoplus_{k \leq n} \MB{C}[\Sy{k}] \rightarrow \MB{C}$ associated to $\lambda$ is defined so that for $x \in \MB{C}[\Sy{k}]$,
\begin{equation} \label{eqn-char-map}
\norcharrep{\lambda}(x) := \frac{\charrep{\lambda}(\symembedding{k}{n}(x))}{\dim \simplerep{\lambda}} = \frac{\charrep{\lambda}(\symembedding{k}{n}(x))}{\charrep{\lambda}(1_n)}.
\end{equation}

Let $\mu = (\mu_1, \dots, \mu_t) \vdash k \leq n$ and set $\partitioncycle{\mu} = 1_k$ if $\mu = (1^k)$ and otherwise
\begin{equation*}
\partitioncycle{\mu} =  \Big(s_{k-1}\dots s_{k-\mu_t+1}\Big) \dots \Big(s_{\mu_1+\mu_2-1} \dots s_{\mu_1+1}\Big) \dots \Big(s_{\mu_1-1} \dots s_{2}s_1\Big)
\end{equation*}
\begin{equation*}
= (k, k-1, \dots, k - \mu_t +1)(\mu_1 + \mu_2, \dots, \mu_1 +1)\dots(\mu_1, \dots,2,1) \in \Sy{k}.
\end{equation*}
We define
\begin{equation*}
\partitioncyclen{\mu}{n} := \longestelement{n}^{-1} (\symembedding{k}{n}(\partitioncycle{\mu})) \longestelement{n} \in \Sy{n}.
\end{equation*}
Observe that $\partitioncyclen{\mu}{n}$ has cycle type $\partembedding{k}{n}(\mu)$ and fixes $1, 2, \dots, n-k$ pointwise.

\begin{example}
Let $\mu = (3,2) \vdash 5$, then
\begin{equation*}
\partitioncycle{\mu} = (s_4)(s_2s_1) = (5,4)(3,2,1)
\end{equation*}
and we see that $\shape{\partitioncycle{\mu}} = \mu$. For $n = 8$,
\begin{equation*}
\partitioncyclen{\mu}{8} = s_4s_6s_7 = (4,5)(6,7,8),
\end{equation*}
while for $n = 10$,
\begin{equation*}
\partitioncyclen{\mu}{10} = (6,7)(8,9,10).
\end{equation*}
\end{example}

The elements
\begin{equation*}
\{1_n,\partitioncyclen{(2)}{n},\partitioncyclen{(3)}{n}, \dots, \partitioncyclen{(n)}{n} \} = \{1_n, \;s_{n-1}, \; s_{n-2}s_{n-1},\; \dots, \;s_1s_2\dots s_{n-1}\}
\end{equation*}
are the minimal length left coset representatives of $\Sy{n-1}$ in $\Sy{n}$. We extend this observation in the following lemma.

\begin{lemma} \label{lemma-cosetreps} For $k < n$, the elements of the set
\begin{equation*}
\{\partitioncyclen{(i_n)}{n}\partitioncyclen{(i_{n-1})}{n-1} \dots \partitioncyclen{(i_{k+1})}{k+1} \; | \; 1 \leq i_j \leq j \; \}
\end{equation*}
are the minimal length left coset representatives of $\Sy{k}$ in $\Sy{n}$. We denote this set by $\speclcos{n}{k}$.
\end{lemma}
\omitt{
\tonytodo{Maybe this should be a Lemma?}
}
We note that $|\speclcos{n}{k}| = \fallingfactorial{n}{n-k}$, where the \emph{falling factorial power} is defined as
\begin{equation*}
\fallingfactorial{x}{k} = \begin{cases}
x(x-1)\dots (x-k+1), & \text{if $k = 1,2, \dots$} \\
1, & \text{if $k = 0$}.
\end{cases}
\end{equation*}

\begin{example}
We have
\begin{equation*}
\speclcos{4}{3} = \{\;1_4,\;{\color{red}{s_3}},\;{\color{red}{s_2s_3}},\;{\color{red}{s_1s_2s_3}} \},
\end{equation*}
\begin{equation*}
\speclcos{3}{2} = \{\;1_3,\;{\color{blue}{s_2}},\;{\color{blue}{s_1s_2}}\},
\end{equation*}
and
\begin{flalign*}
\speclcos{4}{2} =  \{\;&1_4,\;{\color{red}{s_3}},\;{\color{red}{s_2s_3}},\;{\color{red}{s_1s_2s_3}},\\ 
&{\color{blue}{s_2}},\;{\color{red}{s_3}}{\color{blue}{s_2}},\;{\color{red}{s_2s_3}}{\color{blue}{s_2}}, \; {\color{red}{s_1s_2s_3}}{\color{blue}{s_2}}, \\
&{\color{blue}{s_1s_2}}, \; {\color{red}{s_3}}{\color{blue}{s_1s_2}}, \; {\color{red}{s_2s_3}}{\color{blue}{s_1s_2}}, \; {\color{red}{s_1s_2s_3}}{\color{blue}{s_1s_2}} \}.
\end{flalign*}
\end{example}

\subsection{The center of $\MB{C}[\Sy{n}]$}

For $\mu \vdash k \leq n$, set
\begin{equation*}
\congclasssum{\mu}{n} := \sum_{\substack{g \in \Sy{n}, \\  \shape{g} = \partembedding{k}{n}(\mu)}} g.
\end{equation*}
The elements $\{\congclasssum{\mu}{n}\}_{\mu \vdash n}$ are a basis for the center of the symmetric group algebra, $Z(\MB{C}[\Sy{n}])$. We write $z_{\mu,n}$ for the size of the centralizer of an element in $\Sy{n}$ with cycle type $\partembedding{k}{n}(\mu)$. Note that when $\mu \vdash n$, then $z_{\mu,n} = z_{\mu}$.

\begin{definition} \label{def-identity-of-a}
For $\mu = (\mu_1, \dots, \mu_t) \vdash k \leq n$, set
\begin{equation} 
\classsum{\mu}{n} := \sum_{g \in \speclcos{n}{n-k}} g \partitioncyclen{\mu}{n} g^{-1}.
\end{equation}
We call $\classsum{\mu}{n}$ the {\emph{normalized conjugacy class sum}} associated to $\mu$ in $\MB{C}[\Sy{n}]$.
\end{definition}

Alternatively, $\classsum{\mu}{n}$ may be written as
\begin{equation} \label{eqn-alt-a}
\classsum{\mu}{n} = \sum (i_1,\dots, i_{\mu_1}) \dots (i_{k-\mu_t +1}, \dots, i_k)
\end{equation}
where this sum is taken over all distinct $k$-tuples $(i_1, \dots, i_k)$ of elements from $\{1, 2, \dots, n\}$. From \eqref{eqn-alt-a} an easy counting argument shows that
\begin{equation} \label{eqn-identity-of-a}
\classsum{\mu}{n} = \frac{z_{\mu,n}}{(n-k)!} \congclasssum{\mu}{n}.
\end{equation}
It follows from \eqref{eqn-identity-of-a} that $\classsum{\mu}{n} \in Z(\MB{C}[\Sy{n}])$.

\begin{example} \label{ex-single-cycle}
Let $k \leq n$. When $\mu = (k) \vdash k$, then $z_{(k),n} = k(n-k)!$ so that
\begin{equation*}
\classsum{(k)}{n} = k\congclasssum{(k)}{n}.
\end{equation*} 
\end{example} 
 
The elements $\classsum{\mu}{n}$ are important in the study of the asymptotic character theory of symmetric groups \cite{KO94}. They also appear in connection with the algebra of partial permutations \cite{IK99}. If $\mu \vdash k \leq n$ and $\lambda \vdash n$ then
 \begin{equation} \label{prop-normalized-char-class-sum}
 \norcharrep{\lambda}(\classsum{\mu}{n}) = (n \downharpoonright k)\frac{\chi^{\lambda}(\mu)}{\dim\simplerep{\lambda}}.
 \end{equation}
\omitt{
\begin{proof}
A simple counting argument shows that,
\begin{equation*}
\norcharrep{\lambda}(\congclasssum{\mu}{n}) = \frac{n!\charrep{\lambda}(\mu)}{z_{\mu,n}\dim \simplerep{\lambda}}.
\end{equation*}
The result then follows from \eqref{eqn-identity-of-a}.
\end{proof}
}
The following is well-known.

\begin{proposition} \label{prop-normchar-is-homomorphism}
When restricted to $Z(\MB{C}[\Sy{n}])$, the normalized character $\norcharrep{\lambda}$ is an algebra homomorphism from $Z(\MB{C}[\Sy{n}])$ to $\MB{C}$.
\end{proposition}

\omitt{
\begin{proof}
Recall that $\simplerep{\lambda}$ is the simple $\MB{C}[\Sy{n}]$-module associated to $\lambda$. Since $\congclasssum{\mu}{n} \in Z(\MB{C}[\Sy{n}])$, it acts as a scalar on $\simplerep{\lambda}$. It can be checked that for any nonzero $v \in \simplerep{\lambda}$,
\begin{equation*}
\congclasssum{\mu}{n}v = \norcharrep{\lambda}(\mu)v.
\end{equation*}
The result follows.
\end{proof}
}

$Z(\MB{C}[\Sy{n}])$ is also generated by symmetric polynomials in the Jucys-Murphy elements \newline$\{\JM{i}\}_{1 \leq i \leq n} \subseteq \MB{C}[\Sy{n}]$, where
\begin{equation*}
\JM{1} = 0, \quad\quad \text{and} \quad\quad \JM{k} = (1,k) + (2,k) + \dots + (k-1,k), \;\;\;\;\; 2 \leq k \leq n.
\end{equation*}
We can also write
\begin{equation} \label{eqn-alternative-JM-pres}
\JM{k} = \sum_{i = 1}^{k-1} s_i \dots s_{k-2} s_{k-1} s_{k-2} \dots s_i.
\end{equation}


\subsection{The transition measure and co-transition measure} \label{sect-transition}

In this section we recall the notion of transition and co-transition measures, also known as growth and decay, respectively. Assume that $\lambda \vdash n$ and let $\interlacingx{1}, \dots, \interlacingx{d}$ and $\interlacingy{1}, \dots, \interlacingy{d-1}$ be the interlacing sequences associated to $\lambda$. Recall that $\lambda^{(1)}, \dots, \lambda^{(d)}$ are the partitions of $n+1$ such that $\content{\lambda^{(i)} / \lambda} = \interlacingx{i}$, while $\lambda_{(1)}, \dots, \lambda_{(d-1)}$ are the partitions of $n-1$ such that $\content{\lambda / \lambda_{(i)}} = \interlacingy{i}$.

For $1 \leq i \leq d$, the {\emph{transition probabilities}} for $\lambda$ are defined as
\begin{equation*}
\transitionprob{\lambda}{\lambda^{(i)}} := 
\frac{\dim(\simplerep{\lambda^{(i)}})}{(n+1)\dim(\simplerep{\lambda})}. 
\end{equation*}
The {\emph{transition measure}} $\transition{\lambda}$ is then the probability measure on $\MB{R}$ defined by
\begin{equation} \label{eqn-transition-measure-def}
\transition{\lambda} := \sum_{i=1}^d \transitionprob{\lambda}{\lambda^{(i)}} \delta_{\interlacingx{i}}
\end{equation}
where $\delta_{\interlacingx{i}}$ is the Dirac delta measure with support on $\interlacingx{i} \in \MB{R}$. Dually, for $1 \leq i \leq d-1$ the {\emph{co-transition probabilities}} of $\lambda$ are
\begin{equation*}
\cotransitionprob{\lambda}{\lambda_{(i)}} :=
\frac{\dim(\simplerep{\lambda_{(i)}})}{\dim(\simplerep{\lambda})}
\end{equation*}
and the {\emph{co-transition measure}} $\cotransition{\lambda}$ is 
\begin{equation} \label{eqn-cotransition-measure}
\cotransition{\lambda} := \sum_{i=1}^{d-1} \cotransitionprob{\lambda}{\lambda_{(i)}} \delta_{\interlacingy{i}}.
\end{equation}
These probability measures were first investigated by Kerov (\cite{Ker93}, \cite{Ker00}). They are fundamental tools in the study of the asymptotic representation theory of symmetric groups and in the connection between asymptotic representation theory and free probability. 

\omitt{The interlacing sequence approach to Young diagrams is particularly well-suited to study of transition and co-transition measures. For example, Kerov shows in \cite{Ker00} that,
\begin{equation} \label{eqn-interlacing-and-tran-prob}
\sum_{i =1}^d \frac{\transitionprob{\lambda}{\lambda^{(i)}}}{z - \interlacingx{i}} = \frac{(z-\interlacingy{1}) \dots (z - \interlacingy{d-1})}{(z-\interlacingx{1}) \dots (z - \interlacingx{d-1})(z - \interlacingx{d})} 
\end{equation}
and
\begin{equation} \label{eqn-interlacing-and-cotran-prob}
z - |\lambda|\sum_{i =1}^{d-1} \frac{\cotransitionprob{\lambda}{\lambda_{(i)}}}{z - \interlacingy{i}} = \frac{(z-\interlacingx{1}) \dots (z - \interlacingx{d-1})(z - \interlacingx{d})}{(z-\interlacingy{1}) \dots (z - \interlacingy{d-1})}.
\end{equation}
}

The $k$th moment associated to the transition measure $\transition{\lambda}$ is given by
\begin{equation*}
\moment{k}(\lambda) = \sum_{i = 1}^d \interlacingx{i}^k\transitionprob{\lambda}{\lambda^{(i)}} 
\end{equation*}
while the $k$th moment associated to the co-transition measure $\cotransition{\lambda}$ is given by
\begin{equation*}
\comoment{k}(\lambda) = \sum_{i = 1}^{d-1} \interlacingy{i}^k\cotransitionprob{\lambda}{\lambda_{(i)}}.
\end{equation*}


We write the moment generating series for the transition measure (resp. co-transition measure) as 
\begin{equation*}
\tranmomentseries{\lambda} := \sum_{k = 0}^\infty \moment{k}(\lambda) z^{-k-1} \quad \text{and} \quad \cotranmomentseries{\lambda} := z - \sum_{k = 0}^\infty |\lambda|\comoment{k}(\lambda) z^{-k-1}. 
\end{equation*}
Note that we scale all coefficients of $\cotranmomentseries{\lambda}$ by $|\lambda|$ with the exception of the coefficient on $z$.

\begin{lemma} \label{lemma-momentseries-interlacing}
For $\lambda \in \partitions$
\begin{equation} \label{eqn-tranmoment-gen-series-eqty}
\tranmomentseries{\lambda} = (\cotranmomentseries{\lambda})^{-1}.
\end{equation}
\omitt{
and $\interlacingx{1}, \dots, \interlacingx{d}$ and $\interlacingy{1}, \dots, \interlacingy{d-1}$ be the interlacing sequences associated with $\lambda$. Then
\begin{equation} \label{eqn-tranmoment-gen-series-eqty}
\tranmomentseries{\lambda} = \frac{(z-\interlacingy{1})\dots (z-\interlacingy{d-1})}{(z-\interlacingx{1}) \dots (z - \interlacingx{d-1})(z - \interlacingx{d})} = (\cotranmomentseries{\lambda})^{-1}.
\end{equation}
} 
\end{lemma}

\begin{proof}
This follows directly from equation $(2.3)$ and Lemma 5.1 in \cite{Ker00}.
\end{proof}

\omitt{
\begin{proof}
Begin by observing that 
\begin{equation*}
\sum_{i = 1}^d \frac{\transitionprob{\lambda}{\lambda^{(i)}}}{z - \interlacingx{i}} =  \sum^\infty_{k = 0} \sum_{i = 1}^d \transitionprob{\lambda}{\lambda^{(i)}}\interlacingx{i}^kz^{-k-1} = \sum^\infty_{k = 0} \moment{k}(\lambda)z^{-k-1} = \tranmomentseries{\lambda}.
\end{equation*}
Similarly,
\begin{equation*}
z - |\lambda|\sum_{i =1}^{d-1} \frac{\cotransitionprob{\lambda}{\lambda_{(i)}}}{z - \interlacingy{i}} = z - |\lambda|\sum^\infty_{k = 0} \sum_{i = 1}^{d-1} \cotransitionprob{\lambda}{\lambda_{(i)}}\interlacingy{i}^kz^{-k-1} = z - \sum^\infty_{k = 0} |\lambda| \comoment{k}z^{-k-1} =  \cotranmomentseries{\lambda}.
\end{equation*}
The result then follows from \eqref{eqn-interlacing-and-tran-prob} and \eqref{eqn-interlacing-and-cotran-prob}.
\end{proof}
}

The {\emph{boolean cumulants}} $\{\Boolean{k}(\lambda)\}_{k \geq 1}$ associated to $\transition{\lambda}$ can be defined as the coefficients on the multiplicative inverse of $\tranmomentseries{\lambda}$,
\begin{equation} \label{eqn-boolean-definition}
\booleanseries{\lambda} = z - \sum_{k = -1}^\infty \Boolean{k+2}(\lambda)z^{-k-1} = (\tranmomentseries{\lambda})^{-1}.
\end{equation}
With Lemma \ref{lemma-momentseries-interlacing} this definition immediately gives us the following fact.

\begin{proposition} \label{prop-boolean-comoment}
Let $\lambda \in \partitions$ and $k \geq 0$, then $\Boolean{1}(\lambda) = 0$ and
\begin{equation}
\Boolean{k+2}(\lambda) = |\lambda|\comoment{k}(\lambda).
\end{equation}
\end{proposition}

\begin{remark} \label{rmk-recursive-for-boolean-moments}
The equality \eqref{eqn-boolean-definition} can be rewritten as 
\begin{equation} \label{eqn-recursive-rln-moment-boolean}
\sum_{i = 1}^k \moment{k-i}(\lambda)\Boolean{i}(\lambda) = \moment{k}(\lambda).
\end{equation}
\end{remark}

\omitt{
while the Boolean cumulants $\{\coboolean{k}(\lambda)\}_{k \geq 1}$ for $\cotransition{\lambda}$ can be defined as the coefficients of
\begin{equation*}
\cobooleanseries{\mu} = \Big(\frac{z}{|\mu|} - \frac{1}{|\mu|}\cotranmomentseries{\mu}\Big)^{-1} = z - \sum_{k = 1}^\infty \coboolean{k}(\mu)z^{1-k}.
\end{equation*}
}

For general information about the relationship between moments, Boolean cumulants, and other families of cumulants see \cite{AHLV15}.

There is a more algebraic approach to the transition measure due to Biane \cite{B98}. Let 
\begin{equation*}
\pr{n-1}: \MB{C}[\Sy{n}] \rightarrow \MB{C}[\Sy{n-1}] \subset \MB{C}[\Sy{n}]
\end{equation*}
be the projection map defined on $\Sy{n}$ by
\begin{equation*}
\pr{n-1}(g) = \begin{cases}
g & \text{if $g(n) = n$} \\
0 & \text{otherwise.}
\end{cases}
\end{equation*}
In the context of probability theory, $\pr{n-1}$ is sometimes known as the {\emph{conditional expectation}}.

\begin{proposition} \label{prop-alt-definition-moments}
For $\lambda \vdash n$,
\begin{equation} \label{eqn-moment}
\moment{k}(\lambda) = \norcharrep{\lambda}[\pr{n}(\JM{n+1}^k)]
\end{equation}
and
\begin{equation} \label{eqn-comoments-algebraic}
\Boolean{k+2}(\lambda) = |\lambda|\comoment{k}(\lambda) = \norcharrep{\lambda}\Big(\sum_{i=1}^{n} s_i \dots s_{n-1}\JM{n}^ks_{n-1} \dots s_i \Big).
\end{equation}
\end{proposition} 

\begin{proof}
The statement of \eqref{eqn-moment} appears in \cite{Bia03} Section 4. A detailed proof is given in Theorem 9.23 of \cite{HO07}. To get \eqref{eqn-comoments-algebraic} note that since characters are class functions,
\begin{equation*}
\norcharrep{\lambda}\Big(\sum_{i}^{n} s_i \dots s_{n-1}\JM{n}^ks_{n-1} \dots s_i \Big) = |\lambda|\norcharrep{\lambda}(\JM{n}^k).
\end{equation*}
As $\JM{n}$ eigenspaces, $\simplerep{\lambda}$ decomposes as
\begin{equation*}
\simplerep{\lambda} \cong \bigoplus_{i = 1}^{d-1} \simplerep{\lambda_{(i)}}
\end{equation*}
with $\simplerep{\lambda_{(i)}}$ corresponding to eigenvalue $b_i$ \cite{OkV04}. Hence,
\begin{equation*}
|\lambda|\norcharrep{\lambda}(\JM{n}^k) = |\lambda|\sum^{d-1}_{i=1}  \frac{\dim(\lambda_{(i)}) b_i^k}{\dim(\lambda)} = |\lambda|\comoment{k}(\lambda) = \Boolean{k+2}(\lambda).
\end{equation*}

\end{proof}

Proposition \ref{prop-alt-definition-moments} is related to the fact that we are working in a noncommutative probability space (that is, a von Neumann algebra equipped with a normal faithful trace). In our case the algebra is $\End(\simplerep{\lambda})\otimes M_{n+1}(\MB{C})$ and $\transition{\lambda}$ then arises from the distribution of a self-adjoint element in this algebra (see Proposition 3.3 in \cite{B98}).

\omitt{
\tonytodo{Maybe we should add a comment here mentioning the fact that the above proposition is related to the fact that we are working in a noncommutative probability space (that is, a noncommutative von Neumann algebra equipped with a normal faithful trace, or something like that)}}

\section{Symmetric functions and shifted symmetric functions}

In order to define the algebra of shifted symmetric functions, we first recall the classical symmetric functions. Let $\Lambda_n$ be the algebra of symmetric polynomials over $\MB{C}$ in $x_1, \dots, x_n$. This algebra is graded by polynomial degree. Recall that for $n \geq 0$ there is a homomorphism
\begin{equation} \label{eqn-poly-surj}
\Lambda_{n+1} \rightarrow \Lambda_{n}
\end{equation}
given by setting $x_{n+1} = 0$ in $\Lambda_{n+1}$. One can define the algebra of symmetric functions as the projective limit $\Lambda =  \varprojlim \Lambda_n$ taken in the category of graded algebras. We recall three collections of algebraically independent generators of $\Lambda$:
\begin{itemize}
\item elementary symmetric functions $e_1, e_2, e_3, \dots$,
\item complete homogeneous symmetric functions $h_1, h_2, h_3, \dots$,
\item power sum symmetric functions $p_1, p_2, p_3, \dots$
 \end{itemize}
For $\{f_k\}_{k \geq 1}$ equal to any of these three sets of generators and $\lambda = (\lambda_1,\dots,\lambda_r)$ we write $f_\lambda := f_{\lambda_1} \dots f_{\lambda_r}$. We denote the basis of Schur functions by $\{s_\lambda\}_{\lambda \in \partitions}$. We refer the reader to \cite{Mac15} and \cite{RS99} for background on $\Lambda$.

Let $\ShiftSym{n}$ be the algebra of polynomials over $\MB{C}$ in $x_1, \dots, x_n$, which become symmetric in the new variables $x_i' = x_i - i$. This algebra is filtered by polynomial degree. In analogy to $\Lambda_{n+1}$, setting $x_{n+1} = 0$ in $\ShiftSym{n+1}$ gives a homomorphism 
\begin{equation} \label{eqn-shift-poly-surj}
\ShiftSym{n+1} \rightarrow \ShiftSym{n}
\end{equation}
which respects the filtration. Using \eqref{eqn-shift-poly-surj}, set 
\begin{equation*}
\ShiftSym{} := \varprojlim \ShiftSym{n}, 
\end{equation*}
where this limit is taken in the category of filtered algebras. $\ShiftSym{}$ is called the {\emph{algebra of shifted symmetric functions}}. 

Because $\ShiftSym{}$ is filtered, we can consider the associated graded algebra $\assocgraded(\ShiftSym{})$. 

\begin{proposition} \cite[Prop. 1.5]{OO97} \label{prop-grshift}
$\assocgraded(\ShiftSym{})$ is canonically isomorphic to $\Lambda$. 
\end{proposition}

\begin{remark}
It is noted in Remark 1.7 of \cite{OO97} that we may also view $\ShiftSym{}$ as a deformation of $\Lambda$. Let $\ShiftSym{n}(\theta)$ be the algebra of polynomials in $x_1, \dots, x_n$ which are symmetric in the new variables $x'_i = x_i + c - i\theta$ for $1 \leq i \leq n$ and where $c \in \MB{C}$. Define $\ShiftSym{}(\theta) = \varprojlim \ShiftSym{n}(\theta)$. Then $\ShiftSym{}(0) = \Lambda$ and $\ShiftSym{}(1) = \ShiftSym{}$. In fact for all $\theta \neq 0$, $\ShiftSym{}(\theta) \cong \ShiftSym{}$.
\end{remark}

\subsection{Bases of $\ShiftSym{}$}

In \cite{OO97} Okounkov and Olshanski introduced a remarkable basis for $\ShiftSym{}$ called the shifted Schur functions. Let $\lambda = (\lambda_1, \dots, \lambda_n)$ be a partition with $\lambda_1 \geq \dots \geq \lambda_n \geq 0$ (note that here we allow components of a partition to be zero). The \emph{shifted Schur polynomial in $n$ variables, indexed by $\lambda$} is the ratio of two $n \times n$ determinants,
\begin{equation}
s^*_\lambda(x_1,\dots,x_n) = \frac{\det[(x_i + n - i \downharpoonright \lambda_j + n - j)]}{\det[(x_i + n - i \downharpoonright n - j)]},
\end{equation}
where $1 \leq i, j \leq n$. This polynomial belongs to $\ShiftSym{n}$. It is shown in \cite{OO97} that
\begin{equation}
s^*_\lambda(x_1,\dots,x_n,0) = s^*_{\lambda}(x_1,\dots,x_n).
\end{equation}
This implies that for fixed $\lambda$, letting $n \rightarrow \infty$ gives a well-defined element $\shiftschur{\lambda}$ of $\ShiftSym{}$. The elements $\{\shiftschur{\lambda}\}_{\lambda \in \MC{P}} \in \ShiftSym{}$ are called the \emph{shifted Schur functions} and form a basis for $\ShiftSym{}$. There is a linear map $\ShiftSym{} \rightarrow \assocgraded(\ShiftSym{}) \cong \Lambda$ which sends $f \in \ShiftSym{}$ to its top homogeneous component which is an element of $\Lambda$. Under this map
\begin{equation*}
\shiftschur{\lambda} \mapsto s_\lambda
\end{equation*}
or alternatively,
\begin{equation} \label{eqn-schur-is-top-degree}
\shiftschur{\lambda} = s_\lambda + \loworderterms
\end{equation}
where $\loworderterms$ means lower order terms in polynomial degree.

In analogy to the classical case, the {\emph{elementary shifted functions}} can be defined as $\elementaryshift{k} := \shiftschur{(1^k)}$, while the {\emph{complete shifted functions}} can be defined as $\homogenshift{k} := \shiftschur{(k)}$. More explicitly:
\begin{equation*}
\elementaryshift{k}(x_1,x_2,\dots) = \sum_{1 \leq i_1 < \dots < i_k < \infty} (x_{i_1} + k-1)(x_{i_2} + k -2) \dots x_{i_k}
\end{equation*}
and
\begin{equation*}
\homogenshift{k}(x_1,x_2,\dots) = \sum_{1 \leq i_1 \leq \dots \leq i_k < \infty} (x_{i_1} - k+1)(x_{i_2} - k +2) \dots x_{i_k}.
\end{equation*}

Let $\Symtoshift$ be the linear isomorphism $\Symtoshift: \Lambda \rightarrow \ShiftSym{}$ which sends $s_\lambda \mapsto s_\lambda^*$. Define the element $\shiftpwr{\lambda} \in \ShiftSym{}$ to then be
\begin{equation} \label{eqn-shiftpwr}
\shiftpwr{\lambda} := F(p_\lambda),
\end{equation}
where $p_\lambda$ is the power sum symmetric function.  The elements $\shiftpwr{\lambda}$ are one of several shifted analogues of the power sums. For $\lambda \vdash n$, the transition coefficients between the power-sum and Schur bases are given by the character tables of the symmetric group (see \cite{RS99}):
\begin{equation*}
p_\lambda = \sum_{\mu \vdash n} \charrep{\mu}(\lambda) s_\mu.
\end{equation*}
It follows directly from definition \eqref{eqn-shiftpwr} that
\begin{equation} \label{eqn-shiftpower-in-shiftschur}
\shiftpwr{\lambda} = \sum_{\mu \vdash n} \charrep{\mu}(\lambda) \shiftschur{\mu}.
\end{equation}
Note also that by \eqref{eqn-schur-is-top-degree} and \eqref{eqn-shiftpower-in-shiftschur},
\begin{equation} \label{eqn-power-is-top-degree} 
\shiftpwr{\lambda} = p_\lambda + \loworderterms
\end{equation}
Since the power symmetric functions $p_1, p_2, \dots$ are algebraically independent and generate $\Lambda$, it follows from \eqref{eqn-power-is-top-degree} that $\shiftpwr{1}, \shiftpwr{2}, \dots$ are algebraically independent and generate $\ShiftSym{}$. Similarly, since $\{p_\lambda\}_{\lambda \in \partitions}$ is a basis for $\Lambda$, $\{\shiftpwr{\lambda}\}_{\lambda \in \partitions}$ is a basis for $\ShiftSym{}$. For more properties of the basis $\{\shiftpwr{\lambda}\}$ see \cite{IO02}.

\begin{remark}
Let $\lambda = (\lambda_1, \dots, \lambda_r) \vdash n$. While it is true that in $\Lambda$, $p_{\lambda_1} \dots p_{\lambda_r} = p_\lambda$, in general
\begin{equation*}
\shiftpwr{\lambda_1} \dots \shiftpwr{\lambda_r} \neq \shiftpwr{\lambda}.
\end{equation*}
However, by \eqref{eqn-power-is-top-degree} 
\begin{equation*}
\shiftpwr{\lambda_1} \dots \shiftpwr{\lambda_r} = \shiftpwr{\lambda} + \loworderterms
\end{equation*}
\end{remark}


\subsection{$\ShiftSym{}$ as functions on $\partitionsn{}$} \label{sect-shiftsym-as-partition-functions}

Let $\funonyd$ be the algebra of functions from $\partitionsn{}$ to $\MB{C}$ with pointwise multiplication. Viewing $\mu = (\mu_1, \dots, \mu_t) \vdash k$ as the sequence $(\mu_1, \dots, \mu_t, 0, 0, \dots)$, we can evaluate $f \in \ShiftSym{}$ on $\mu$ by setting
\begin{equation} \label{eqn-how-to-eval-part-on-shift}
f(\mu) = f(\mu_1, \dots, \mu_t, 0, 0, \dots).
\end{equation}
Since $(\mu_1, \dots, \mu_t, 0, 0, \dots)$ has only a finite number of nonzero values, it is clear that \eqref{eqn-how-to-eval-part-on-shift} is well-defined. In fact $f$ is uniquely defined by its values on $\partitionsn{}$. Thus $\ShiftSym{}$ may be realized as a subalgebra of $\funonyd$. This fact is used repeatedly en route to establishing many of the fundamental results about shifted symmetric functions in \cite{KO94} and \cite{OO97}.

For $\lambda \vdash n$ and $\alpha$ a cell in the Young diagram corresponding to $\lambda$ let $h(\alpha)$ be the hook length of $\alpha$. Then set $H(\lambda)$ as the product of all hooklengths in $\lambda$,
\begin{equation*}
H(\lambda) := \prod_{\alpha \in \lambda} h(\alpha).
\end{equation*}
The following is known as the ``Characterization Theorem'' of \cite{O96}.
\begin{theorem} 
For $\mu \vdash k$, $\shiftschur{\mu}$ is the unique element of $\ShiftSym{}$ such that $\deg(\shiftschur{\mu}) \leq k$ and
\begin{equation*}
\shiftschur{\mu}(\lambda) = \delta_{\mu \lambda}H(\mu)
\end{equation*}
for all $\lambda \in \partitions$ such that $|\lambda| \leq |\mu|$.
\end{theorem}

This theorem along with \eqref{eqn-shiftpower-in-shiftschur} then give the following proposition.
\begin{proposition} \label{prop-value-prwshift} \cite{OO97}
For $\mu \vdash k$, $\lambda \vdash n$, 
\begin{equation}
\shiftpwr{\mu}(\lambda) = \begin{cases}
\frac{(n \downharpoonright k)}{\dim \simplerep{\lambda}}\charrep{\lambda}(\mu) & k \leq n\\
0 & \text{otherwise.}\\
\end{cases}
\end{equation}
\end{proposition}

\begin{remark}
We will later use the fact that $\shiftpwr{1} = x_1 + x_2 + \dots = p_1$, so that $\shiftpwr{1}(\lambda) = |\lambda|$ for all $\lambda \in \partitions$.
\end{remark}

In Section \ref{sect-transition} we introduced the moments $\{\moment{k}(\lambda)\}$ (resp. $\{\comoment{k}(\lambda)\}$) of the transition measure (resp. co-transition measure) associated to a partition $\lambda$ and the corresponding Boolean cumulants $\{\Boolean{k}(\lambda)\}$\omitt{(resp. $\{\coboolean{k}\}$)}. We can interpret all of these as elements of $\funonyd$ via
\begin{equation*}
\lambda \xmapsto{\moment{k}} \moment{k}(\lambda), \quad\quad \lambda \xmapsto{\comoment{k}} \comoment{k}(\lambda), \quad \text{and} \quad \lambda \xmapsto{ \Boolean{k} } \Boolean{k}(\lambda).
\end{equation*}
We omit the partition argument from $\moment{k}$, $\comoment{k}$, and $\Boolean{k}$ in this context to emphasize that we are considering them as elements of $\funonyd$. 

\begin{proposition} \cite[Theorem 6.4]{L09} \label{prop-moments-as-sym}
As elements of $\funonyd$, $\moment{k}$ and $\Boolean{k}$ belong to $\ShiftSym{}$.
\end{proposition}

\begin{remark} \label{remark-moments-as-sym}
In \cite{L09} Section 5, Lassalle shows that with the appropriate alphabet $A_\lambda$ (which is specific to each partition $\lambda$), 
\begin{equation} \label{eqn-elementary-homogeneous}
\moment{k}(\lambda) = h_k(A_\lambda) \quad\quad\quad \text{and} \quad\quad\quad  \Boolean{k}(\lambda) = (-1)^{k-1}e_k(A_\lambda).
\end{equation}
\end{remark}

\omitt{
\begin{proposition} \cite[Section 5]{L09} \label{prop-moments-as-sym}
As elements of $\funonyd$, $\moment{k}$ and $\Boolean{k}$ belong to $\ShiftSym{}$ with
\begin{equation} \label{eqn-elementary-homogeneous}
\contentmap(h_k) = \moment{k} \quad\quad\quad \text{and} \quad\quad\quad \contentmap(e_k) = (-1)^{k-1}\Boolean{k}.
\end{equation}
\end{proposition}
}

\omitt{
The recursive relationship in Remark \ref{rmk-recursive-for-boolean-moments} can then be seen to follow from the relationship between elementary and homogeneous symmetric functions.}

\omitt{
It is well known that as functions on partitions, the moments and Boolean cumulants for the transition measure are shifted symmetric functions (see \cite{L13} and \cite{L09}). By Proposition \ref{prop-boolean-comoment}, the moments of the co-transition measure is also belong to $\ShiftSym{}$. It is shown in \cite{L13} that in fact,
\begin{equation*}
\moment{k}(\lambda) = h_k(\contentsalpha{\lambda}) \quad\quad\quad \text{and} \quad\quad\quad \Boolean{k}(\lambda) = (-1)^{k-1}e_k(\contentsalpha{\lambda}).
\end{equation*} 
In other words we have
\begin{equation*}
\contentmap(h_k) = \moment{k} \quad\quad\quad \text{and} \quad\quad\quad \contentmap(e_k) = \Boolean{k}.
\end{equation*}
}


\section{The Heisenberg category $\Heisencat$} 

In \cite{Kho14}, Khovanov defined an additive $\MB{C}$-linear monoidal category $\Heisencat$ which we will call the {\emph{Heisenberg category}}. The objects in $\Heisencat$ are generated by two objects $Q_+$ and $Q_-$. Following the notation of \cite{Kho14}, we denote $Q_{\epsilon_1} \otimes \cdots \otimes Q_{\epsilon_m}$ by $Q_\epsilon$ where $\epsilon = \epsilon_1 \ldots \epsilon_m$ is a finite sequence of pluses and minuses. The unit object, $\UnitModule$, corresponds to the empty sequence $Q_\emptyset$.

The collection of morphisms $\Hom_{\Heisencat}(Q_\epsilon, Q_{\epsilon'})$, for two sequences $\epsilon$ and $\epsilon'$ is the $\C$-vector space spanned by planar diagrams modulo some local relations. The diagrams are oriented compact 1-manifolds embedded in the strip $\MB{R} \times [0,1]$, modulo rel boundary isotopies. The endpoints of the 1-manifolds are located at $\{1, \ldots, m \} \times \{ 0 \}$ and $\{ 1, \ldots, n \} \times \{ 1 \}$, where $m$ and $n$ are the lengths of $\epsilon$ and $\epsilon'$, respectively. Further, the orientation of the 1-manifold at the endpoints must match the signs in the sequences $\epsilon$ and $\epsilon'$. Triple intersections are not allowed. 
\newpage
\begin{example}
The diagram 
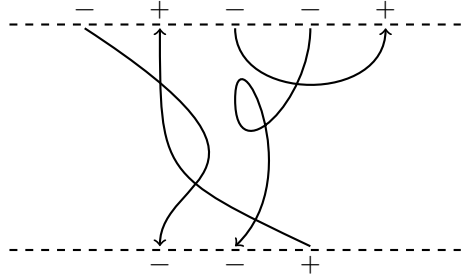
\begin{center}
        \begin{tikzpicture}
        	\draw[dashed, thick] (0,0) -- (6,0);
        	\draw[dashed, thick] (0,3) -- (6,3);
        	
        	\node at (2,-0.2) {$-$};
        	\node at (3,-0.2) {$-$};
        	\node at (4,-0.2) {$+$};
        	
        	\node at (1,3.2) {$-$};
        	\node at (2,3.2) {$+$};
        	\node at (3,3.2) {$-$};
        	\node at (4,3.2) {$-$};
        	\node at (5,3.2) {$+$};
        	
        	\draw[thick,->] (4,0.05) .. controls (2,1) .. (2,2.95);
        	\draw[thick,->] (4,2.95) .. controls (4,2) and (3,1) .. (3,2) .. controls (3,3) and (4,1) .. (3,0.05);
        	\draw[thick,->] (1,2.95) .. controls (4,1) and (2,1) .. (2,0.05);
        	\draw[thick,->] (3,2.95) .. controls (3,1.95) and (5,1.95) .. (5,2.95);
        \end{tikzpicture}
\end{center}
is a morphism from $Q_{--+}$ to $Q_{-+--+}$.
\end{example}
The composition of two morphisms is achieved by stacking diagrams. The local relations for diagrams are:
\vspace{6mm}
\begin{equation} \label{up down double crossings}
        \begin{tikzpicture}[baseline=(current bounding box.center)]
        	
        	\draw[thick,->] (0,0) .. controls (1,1) .. (0,2);
        	\draw[thick,->] (1,2) .. controls (0,1) .. (1,0);
        	
        	\node at (2,1) {$=$};
        	
        	\draw[thick,->] (3,0) -- (3,2);
        	\draw[thick,->] (4,2) -- (4,0);
        	
        	
        	\draw[thick,->] (6,2) .. controls (7,1) .. (6,0);
        	\draw[thick,->] (7,0) .. controls (6,1) .. (7,2);
        	
        	\node at (8,1) {$=$};
        	
        	\draw[thick,->] (9,2) -- (9,0);
        	\draw[thick,->] (10,0) -- (10,2);
        	
        	\node at (11,1) {$-$};
        	
        	\draw[thick,->] (12,2) arc (180:360:0.75);
        	\draw[thick,->] (13.5,0) arc (0:180:0.75);
        \end{tikzpicture}
    \end{equation}
    
    \vspace{4mm}
    
    \begin{equation} \label{anti-clockwise and left curl}
        \begin{tikzpicture}[baseline=(current bounding box.center)]
        	\draw[<-,thick] (0,1) arc (180:-180:.5);
        	
        	\node at (1.5,1) {$= 1$};
        	
        	\draw[thick] (4,1) .. controls (4,1.5) and (4.7,1.5) .. (4.9,1);
          	\draw[thick] (4,1) .. controls (4,0.5) and (4.7,0.5) .. (4.9,1);
          	\draw[thick] (5,0) .. controls (5,0.5) .. (4.9,1) ;
          	\draw[thick,->] (4.9,1) .. controls (5,1.5) .. (5,2);
          	
        	\node at (5.5,1) {$= 0$};
        \end{tikzpicture}
    \end{equation}
    
    \vspace{4mm}
    
    \begin{equation} \label{eqn-symmetric-group-relations}
        \begin{tikzpicture}[baseline=(current bounding box.center)]
        	
        	\draw[->,thick] (0,0) .. controls (1,1) .. (0,2);
        	\draw[->,thick] (1,0) .. controls (0,1) .. (1,2);
        	
        	\node at (2,1) {$=$};
        	
        	\draw[->,thick] (3,0) -- (3,2);
        	\draw[->,thick] (4,0) -- (4,2);
        	
        	
        	\draw[->,thick] (6,0) -- (8,2);
        	\draw[->,thick] (8,0) -- (6,2);
        	
        	\draw[->,thick] (7,0) .. controls (8,1) .. (7,2);
        	
        	\node at (9,1) {$=$};
        	
        	\draw[->,thick] (10,0) -- (12,2);
        	\draw[->,thick] (12,0) -- (10,2);
        	
        	\draw[->,thick] (11,0) .. controls (10,1) .. (11,2);	
	\node at (12.3,0) {.};
        \end{tikzpicture}
\end{equation}

\vspace{6mm}

The relations \eqref{up down double crossings} and \eqref{anti-clockwise and left curl} are motivated by the Heisenberg relation $pq = qp + 1$, where $p$ and $q$ are the two generators of the Heisenberg algebra, while the relations \eqref{eqn-symmetric-group-relations} are motivated by the symmetric group relations.

It is convenient to denote a right curl by a dot on a strand, and a sequence of $d$ right curls by a dot with a $d$ next to it:
\begin{center}
        \begin{tikzpicture}
        	\draw[thick,->] (0,0) -- (0,2);
        	\draw (0,1) node {\textbullet};
        	
        	\draw (1,1) node {$:=$};
        	
        	\draw[thick] (3,1) .. controls (3,1.5) and (2.3,1.5) .. (2.1,1);
        	\draw[thick] (3,1) .. controls (3,0.5) and (2.3,0.5) .. (2.1,1);
        	\draw[thick] (2,0) .. controls (2,0.5) .. (2.1,1);
        	\draw[thick] (2.1,1) .. controls (2,1.5) .. (2,2) [->];
        	
	\node at (3.4,0) {,};
	
        	\draw[->,thick] (5,0) -- (5,2);
        	\draw (5,1) node {\textbullet};
        	\draw (4.7,1) node {$d$};
        	
        	\draw (6,1) node {$:=$};
        	
        	\draw[->,thick] (7,0) -- (7,2);
        	
        	\draw (7,0.4) node {\textbullet};
        	\draw (7,0.8) node {\textbullet};
        	\draw (7,1.2) node {\textbullet};
        	\draw (7,1.6) node {\textbullet};
        	
        	\draw (8.4,1) node {$d$ dots};
	
	\draw [decorate,decoration={brace,amplitude=10pt},xshift=-4pt,yshift=0pt]
(7.4,1.7) -- (7.4,.3) node [black,midway,xshift=0cm,yshift = .6cm] 
{};
	
	\node at (8,0) {.};
        	
        \end{tikzpicture}
\end{center}

A right curl can be moved across intersection points, according to the following ``dot-sliding relations'' \cite{Kho14}:
    \begin{align*}
    	\begin{tikzpicture}[baseline=(current bounding box.center)]
    		\draw[thick,->] (0,0) to (1,1);
    		\draw[thick,->] (1,0) to (0,1);
    		\node at (0.25,0.75) {\textbullet};
    	\end{tikzpicture}
    		\quad &= \quad 
    	\begin{tikzpicture}[baseline=(current bounding box.center)]
    		\draw[thick,->] (0,0) to (1,1);
    		\draw[thick,->] (1,0) to (0,1);
    		\node at (0.75,0.25) {\textbullet};
    	\end{tikzpicture}
    		\quad + \quad
    	\begin{tikzpicture}[baseline=(current bounding box.center)]
    		\draw[thick,->] (0,0) to (0,1);
    		\draw[thick,->] (1,0) to (1,1);
		\node at (1.2,0) {,};
    	\end{tikzpicture} 
    		\\ \\
    	\begin{tikzpicture}[baseline=(current bounding box.center)]
    		\draw[thick,->] (0,0) to (1,1);
    		\draw[thick,->] (1,0) to (0,1);
    		\node at (0.25,0.25) {\textbullet};
    	\end{tikzpicture}
    		\quad &= \quad 
    	\begin{tikzpicture}[baseline=(current bounding box.center)]
    		\draw[thick,->] (0,0) to (1,1);
    		\draw[thick,->] (1,0) to (0,1);
    		\node at (0.75,0.75) {\textbullet};
    	\end{tikzpicture}
    		\quad + \quad
    	\begin{tikzpicture}[baseline=(current bounding box.center)]
    		\draw[thick,->] (0,0) to (0,1);
    		\draw[thick,->] (1,0) to (1,1);
		\node at (1.2,0) {.};
    	\end{tikzpicture}
    \end{align*}

This observation easily generalizes to

    \begin{equation} \label{k dots 1}
    	\begin{tikzpicture}[baseline=(current bounding box.center)]
    		\draw[thick,->] (0,0) to (1,1);
    		\draw[thick,->] (1,0) to (0,1);
    		\node at (0.25,0.75) {\textbullet};
    		\node at (0.5,0.85) {$k$};
    	\end{tikzpicture} 
			\quad = \quad 
    	\begin{tikzpicture}[baseline=(current bounding box.center)]
    		\draw[thick,->] (0,0) to (1,1);
    		\draw[thick,->] (1,0) to (0,1);
    		\node at (0.75,0.25) {\textbullet};
    		\node at (1,0.35) {$k$};
    	\end{tikzpicture}
    		\quad + \sum_{i = 0}^{k-1} \quad
    	\begin{tikzpicture}[baseline=(current bounding box.center)]
    		\draw[thick,->] (0,0) to (0,1);
    		\draw[thick,->] (1,0) to (1,1);
    		\node at (0,0.5) {\textbullet};
    		\node at (1,0.5) {\textbullet};
    		\node at (-0.3,0.5) {$i$};
    		\node at (2,0.5) {$k-1-i$};
		\node at (3,0) {,};
    	\end{tikzpicture}
	\end{equation}
	
	\begin{equation} \label{k dots 2}
    	\begin{tikzpicture}[baseline=(current bounding box.center)]
    		\draw[thick,->] (0,0) to (1,1);
    		\draw[thick,->] (1,0) to (0,1);
    		\node at (0.25,0.25) {\textbullet};
    		\node at (0,0.45) {$k$};
    	\end{tikzpicture}
    		\quad = \quad 
    	\begin{tikzpicture}[baseline=(current bounding box.center)]
    		\draw[thick,->] (0,0) to (1,1);
    		\draw[thick,->] (1,0) to (0,1);
    		\node at (0.75,0.75) {\textbullet};
    		\node at (0.55,0.95) {$k$};
    	\end{tikzpicture}
    		\quad + \sum_{i = 0}^{k-1} \quad
    	\begin{tikzpicture}[baseline=(current bounding box.center)]
    		\draw[thick,->] (0,0) to (0,1);
    		\draw[thick,->] (1,0) to (1,1);
    		\node at (0,0.5) {\textbullet};
    		\node at (1,0.5) {\textbullet};
    		\node at (-0.3,0.5) {$i$};
    		\node at (2,0.5) {$k-1-i$};
		\node at (3,0) {.};
    	\end{tikzpicture} 
    \end{equation}

Another consequence of relations \eqref{up down double crossings}-\eqref{eqn-symmetric-group-relations} are the ``bubble moves'' \cite{Kho14}:

\begin{equation} \label{bubble clockwise}
	\begin{tikzpicture}[baseline=(current bounding box.center)]
		\draw[thick,->] (0,0) arc (180:0:0.5);
		\draw[thick](1,0) arc (0:-180:0.5);
		\node at (0,0) {\textbullet};
		\node at (-0.3,0) {$k$};
		
		\draw[thick,->] (1.5,-1) to (1.5,1);
	\end{tikzpicture}
	\enskip = \enskip
	\begin{tikzpicture}[baseline=(current bounding box.center)]
		\draw[thick,->] (0,0) arc (180:0:0.5);
		\draw[thick](1,0) arc (0:-180:0.5);
		\node at (0,0) {\textbullet};
		\node at (-0.3,0) {$k$};
			
		\draw[thick,->] (-0.8,-1) to (-0.8,1);
	\end{tikzpicture}
		\;\;\; + \;\;\; (k+1) \enskip
	\begin{tikzpicture}[baseline=(current bounding box.center)]
		\draw[thick,->] (0,0) to (0,2);
		\node at (0,1) {\textbullet};
		\node at (0.3,1) {$k$};
	\end{tikzpicture}
		\;\;\; - \;\;\; \sum_{i = 0}^{k-2} (k - i - 1)
	\begin{tikzpicture}[baseline=(current bounding box.center)]
		\draw[thick,->] (0,0) to (0,2);
		\node at (0,1) {\textbullet};
		\node at (1,1) {$k-i-2$};
	\end{tikzpicture}
		\enskip
	\begin{tikzpicture}[baseline=(current bounding box.center)]
		\draw[thick,->] (0,0) arc (180:0:0.5);
		\draw[thick](1,0) arc (0:-180:0.5);
		\node at (0,0) {\textbullet};
		\node at (-0.3,0) {$i$};
		\node at (1.4,-1) {,};
		\node at (1.4,1) {};
	\end{tikzpicture}
\end{equation}
	
\begin{equation} \label{bubble anticlockwise}
	\begin{tikzpicture}[baseline=(current bounding box.center)]
		\draw[thick,->] (0,0) arc (-180:0:0.5);
		\draw[thick](1,0) arc (0:180:0.5);
		\node at (0,0) {\textbullet};
		\node at (-0.3,0) {$k$};
			
		\draw[thick,->] (-0.8,-1) to (-0.8,1);
	\end{tikzpicture}
		\enskip = \enskip
	\begin{tikzpicture}[baseline=(current bounding box.center)]
		\draw[thick,->] (0,0) arc (-180:0:0.5);
		\draw[thick](1,0) arc (0:180:0.5);
		\node at (0,0) {\textbullet};
		\node at (-0.3,0) {$k$};
			
		\draw[thick,->] (1.5,-1) to (1.5,1);
	\end{tikzpicture}
		 \;\;\; - \;\;\; \sum_{i = 0}^{k-2} (k - i - 1)
	\begin{tikzpicture}[baseline=(current bounding box.center)]
		\draw[thick,->] (0,0) arc (180:0:0.5);
		\draw[thick](1,0) arc (0:-180:0.5);
		\node at (0,0) {\textbullet};
		\node at (-0.3,0) {$i$};
	\end{tikzpicture}
		\enskip
	\begin{tikzpicture}[baseline=(current bounding box.center)]
		\draw[thick,->] (0,0) to (0,2);
		\node at (0,1) {\textbullet};
		\node at (1,1) {$k-i-2$};
		\node at (1.8,0) {.};
		\node at (1.4,0) {};
	\end{tikzpicture}
\end{equation}

Note that relations \eqref{eqn-symmetric-group-relations} imply that there is a homomorphism $\SymtoHplus{n}: \MB{C}[\Sy{n}] \rightarrow \EndUpStrands{n}$ which sends
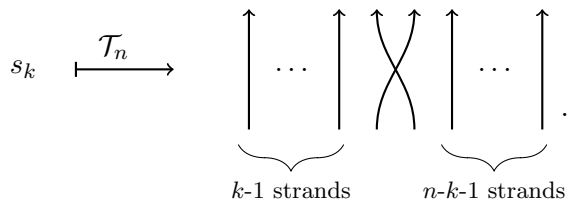
\vspace{1mm}
\begin{center}
\begin{tikzpicture}

\draw[->,thick] (-.3,0)--(1,0);
\draw[thick] (-.3,.1)--(-.3,-.1);
\node at (.2,.3) {$\SymtoHplus{n}$};
\node at (-1,0) {$s_k$};

\draw[thick,->] (2,-.8) -- (2,.8);
\draw[thick,->] (3.2,-.8) -- (3.2,.8);
\draw[thick,->] (3.7,-.8) to [out =90,in = 270] (4.2,.8);
\draw[thick,->] (4.2,-.8) to [out =90,in = 270] (3.7,.8);
\draw[thick,->] (4.7,-.8) -- (4.7,.8);
\draw[thick,->] (5.9,-.8) -- (5.9,.8);
\node at (6.2,-.6) {$.$};

\node at (2.6,0) {\dots};
\node at (5.3,0) {\dots};

\draw [decorate,decoration={brace,amplitude=10pt},xshift=-4pt,yshift=0pt]
(3.4,-1) -- (2,-1) node [black,midway,xshift=0cm,yshift = -.6cm] 
{\footnotesize $k$-$1$ strands};
\draw [decorate,decoration={brace,amplitude=10pt},xshift=-4pt,yshift=0pt]
(6.1,-1) -- (4.7,-1) node [black,midway,xshift=0cm,yshift = -.6cm] 
{\footnotesize $n$-$k$-$1$ strands};

\end{tikzpicture}
\end{center}
Diagrammatically, for $x \in \MB{C}[\Sy{n}]$ we set
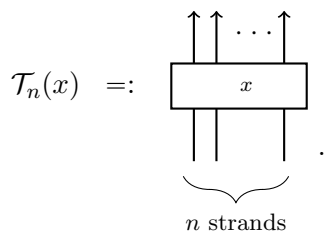
\begin{center} \label{picture-sym-group-emedding}
\begin{tikzpicture}

\node at (-3.5,0) {$\SymtoHplus{n}(x) \;\; =: \;\;$};

\draw[thick,->] (-1.7,-1)--(-1.7,1);
\draw[thick,->] (-.8,-1)--(-.8,1);
\draw[thick,->] (-2,-1)--(-2,1);


\node at (-1.4,.7) {$\cdot$};
\node at (-1.2,.7) {$\cdot$};
\node at (-1.0,.7) {$\cdot$};

\draw[fill=white,thick] (-2.3,-.3) rectangle (-.5,.3);

\node at (-1.3,0) {\scriptsize{$x$}};

\node at (-.3,-.9) {.};

\draw [decorate,decoration={brace,amplitude=10pt},xshift=-4pt,yshift=0pt]
(-.6,-1.2) -- (-2,-1.2) node [black,midway,xshift=0cm,yshift = -.6cm] 
{\footnotesize $n$ strands};

\end{tikzpicture}
\end{center}

The appearance of the group algebra $\MB{C}[\Sy{n}]$ as endomorphisms in $\Heisencat$ is responsible for the connection between $\Heisencat$ and the representation theory of symmetric groups.

\omitt{$\SymtoHplus{n}$ is the since it provides a way of identifying idempotents in $\EndUpStrands{n}$ via Young idempotents in $\MB{C}[\Sy{n}]$. Recall that the Karoubi envelope of a category $\MC{C}$ is the category $\Karoubi{\MC{C}}$ whose objects are pairs $(P,e)$ where $P$ is an object in $\MC{C}$ and $e$ is an idempotent $e: P \rightarrow P$ so that $e^2 = e$. Set $\MC{H} := \Karoubi{\Heisencat}$. The motivation for studying $\Heisencat$ is that $\MC{H}$ conjecturally categorifies an integral form of the Heisenberg algebra $\integralHeisen$, meaning that it is expected that there is an isomorphism $\integralHeisen \rightarrow K_0(\MC{H})$ with $K_0(\MC{H})$ the Grothendieck group of $\MC{H}$. In \cite{Kho14} an injective homomorphism is constructed from $\integralHeisen$ to $K_0(\MC{H})$, but it has yet to be shown that this map is surjective.}


\subsection{The endomorphism algebra $\Hcenter$}

Let $\Hcenter$ denote the center of $\Heisencat$, that is, the algebra of endomorphisms of the monoidal unit object $\UnitModule$. Diagrammatically, the algebra $\Hcenter$ is the commutative $\MB{C}$-algebra spanned by all closed diagrams, with multiplication given by juxtaposition of diagrams.  The algebra structure of $\Hcenter$ was determined by Khovanov in \cite{Kho14}.  Let $\MB{C}[c_0,c_1,c_2, \dots]$ be the polynomial algebra in countably many indeterminants $\{c_i\}_{i \geq 0}$. 

\begin{theorem} \label{Thm-End-Iso} \cite[Prop. 3]{Kho14}
The map $\psi_0: \MB{C}[c_0, c_1, \dots ] \rightarrow \Hcenter$ which sends 
\begin{equation} \label{eqn-what-is-c_k}
\begin{tikzpicture}

\node at (-1.1,2) {$c_k$};

\draw[->,thick] (-.5,2)--(.8,2);
\draw[thick] (-.5,1.9)--(-.5,2.1);
\node at (.2,2.5) {$\psi_0$};

\draw[thick] (2,2) circle (.5cm);

\draw[thick] (1.5,2) -- (1.6,1.9);
\draw[thick] (1.5,2) -- (1.4,1.9);

\draw[fill=black] (1.65,2.35) circle (.08cm);

\node at (1.4,2.7) {$k$};

\end{tikzpicture}
\end{equation}
is an algebra isomorphism.
\end{theorem}

Henceforth we will freely identify $c_k$ with its image in $\Hcenter$. Another natural set of diagrams to consider are the counterclockwise-oriented circles with $k$ right-twist curls on them. Set
\begin{center}
\begin{tikzpicture}
\node at (6.3,2) {$\ctilde{k} \;\;\; := \;\;\;$};

\draw[thick] (8,2) circle (.5cm);

\draw[thick] (7.5,2) -- (7.6,2.1);
\draw[thick] (7.5,2) -- (7.4,2.1);

\draw[fill=black] (7.6,2.35) circle (.08cm);

\node at (7.3,2.9) {$k$};

\node at (8.8,1.6) {$.$};
\end{tikzpicture}
\end{center}
It follows from the relations in \eqref{anti-clockwise and left curl} that $\ctilde{0} = 1$ and $\ctilde{1} = 0$. 

\begin{lemma} \cite[Prop. 2]{Kho14} \label{lemma-c-ctilde-rln}
For $k > 0$, 
	\begin{equation} \label{anticlockwise to clockwise}
		\ctilde{k+1} = \sum_{i = 0}^{k-1} \tilde{c}_i c_{k-1-i}.
	\end{equation}
\end{lemma}

The final class of elements in $\Hcenter$ we consider are those arising from the closure of permutations (that is, closures of morphisms in the image of $\SymtoHplus{n}$). We define

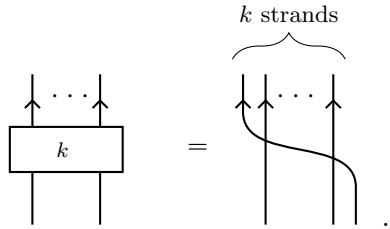
\begin{center}
\begin{tikzpicture}

\draw[thick] (-1.7,-1)--(-1.7,1);
\draw[thick] (-.8,-1)--(-.8,1);

\node at (-1.7,.6) {\arrowlines};
\node at (-.8,.6) {\arrowlines};

\node at (-1.4,.7) {$\cdot$};
\node at (-1.2,.7) {$\cdot$};
\node at (-1.0,.7) {$\cdot$};

\draw[fill=white,thick] (-2,-.3) rectangle (-.5,.3);

\node at (-1.3,0) {\scriptsize{$k$}};

\node at (.5,0) {$=$};

\draw [decorate,decoration={brace,amplitude=10pt},xshift=-4pt,yshift=0pt]
(1.1,1.2) -- (2.6,1.2) node [black,midway,xshift=0cm,yshift = .6cm] 
{\footnotesize $k$ strands};

\draw[thick] (2.3,-1)--(2.3,1);
\draw[thick] (1.4,-1)--(1.4,1);
\draw[thick] (2.6,-1) to [out =90,in = 270] (2.6,-.5) to [out = 90, in = 270] (1.1,.5) to [out = 90, in =270] (1.1,1);

\node at (2.3,.6) {\arrowlines};
\node at (1.4,.6) {\arrowlines};
\node at (1.1,.6) {\arrowlines};

\node at (2,.7) {$\cdot$};
\node at (1.8,.7) {$\cdot$};
\node at (1.6,.7) {$\cdot$};

\node at (3,-1) {$.$};

\end{tikzpicture}
\end{center}

For $\lambda = (\lambda_1,\dots, \lambda_{r}) \vdash n$, let

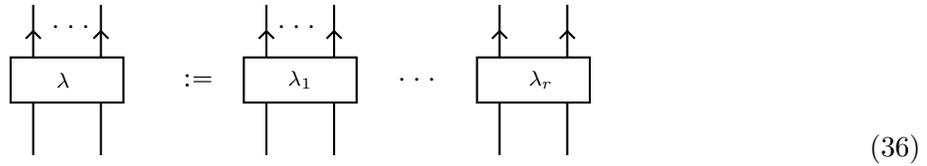
\begin{equation} \label{tikzpicture-def-for-lambda}
\begin{tikzpicture}

\draw[thick] (-1.7,-1)--(-1.7,1);
\draw[thick] (-.8,-1)--(-.8,1);

\node at (-1.7,.6) {\arrowlines};
\node at (-.8,.6) {\arrowlines};

\node at (-1.4,.7) {$\cdot$};
\node at (-1.2,.7) {$\cdot$};
\node at (-1.0,.7) {$\cdot$};

\draw[fill=white,thick] (-2,-.3) rectangle (-.5,.3);

\node at (-1.3,0) {\scriptsize{$\lambda$}};

\node at (.5,0) {$:=$};

\draw[thick] (2.3,-1)--(2.3,1);
\draw[thick] (1.4,-1)--(1.4,1);

\draw[fill=white,thick] (2.6,-.3) rectangle (1.1,.3);

\node at (2.3,.6) {\arrowlines};
\node at (1.4,.6) {\arrowlines};

\node at (2,.7) {$\cdot$};
\node at (1.8,.7) {$\cdot$};
\node at (1.6,.7) {$\cdot$};

\node at (3.2,0) {$\cdot$};
\node at (3.4,0) {$\cdot$};
\node at (3.6,0) {$\cdot$};

\draw[thick] (4.5,-1)--(4.5,1);
\draw[thick] (5.4,-1)--(5.4,1);

\node at (4.5,.6) {\arrowlines};
\node at (5.4,.6) {\arrowlines};

\draw[fill=white,thick] (4.2,-.3) rectangle (5.7,.3);

\node at (5.05,0) {\scriptsize{$\lambda_r$}};
\node at (1.85,0) {\scriptsize{$\lambda_1$}};

\end{tikzpicture}
\end{equation}

then we define

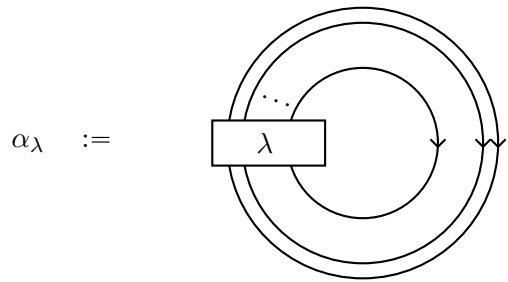
\begin{center}
\begin{tikzpicture}

\node at (-4,0) {$\alpha_\lambda \quad := $};

\draw[thick] (0,0) circle (1.8 cm);
\draw[thick] (0,0) circle (1.6 cm);
\draw[thick] (0,0) circle (1 cm);

\draw[fill=white,thick] (-2,-.3) rectangle (-.5,.3);

\node at (-1,.5) {$\cdot$};
\node at (-1.15,.55) {$\cdot$};
\node at (-1.3,.6) {$\cdot$};

\node at (-1.3,0) {$\lambda$};

\node[rotate = 180] at (1.8,0) {\arrowlines};
\node[rotate = 180] at (1.6,0) {\arrowlines};
\node[rotate = 180] at (1,0) {\arrowlines};

\end{tikzpicture}
\end{center}
with $\alpha_k := \alpha_{(k)}$. 

Lemma \ref{lemma-only-cycle-type-matters} below shows that we could replace the permutation in \eqref{tikzpicture-def-for-lambda} by the image under $\SymtoHplus{n}$ of any $g \in \Sy{n}$ such that $\shape{g} = \lambda$. We choose \eqref{tikzpicture-def-for-lambda} because it will be convenient for later calculations.

\newpage

\begin{lemma} \label{lemma-only-cycle-type-matters}
Suppose that $g_1, g_2 \in \Sy{n}$ are conjugate, so that $\shape{g_1} = \shape{g_2}$. Then
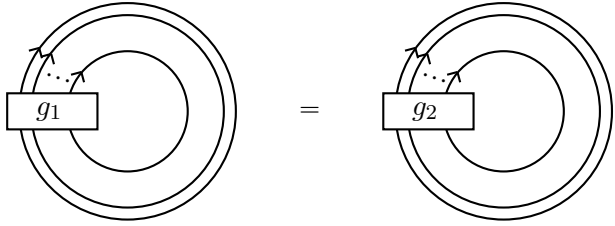
\begin{center}
\begin{tikzpicture}

\node at (0,0) {\begin{tikzpicture}[scale = .8]

\draw[thick] (0,0) circle (1.8 cm);
\draw[thick] (0,0) circle (1.6 cm);
\draw[thick] (0,0) circle (1 cm);

\draw[fill=white,thick] (-2,-.3) rectangle (-.5,.3);

\node at (-1,.5) {$\cdot$};
\node at (-1.15,.55) {$\cdot$};
\node at (-1.3,.6) {$\cdot$};

\node at (-1.3,0) {$g_1$};

\node[rotate = -30] at (-1.5,1) {\arrowlines};
\node[rotate = -30] at (-1.3,.9) {\arrowlines};
\node[rotate = -30] at (-.8,.6) {\arrowlines};
\end{tikzpicture}};

\node at (2.5,0) {$=$};

\node at (5,0) {\begin{tikzpicture}[scale = .8]

\draw[thick] (0,0) circle (1.8 cm);
\draw[thick] (0,0) circle (1.6 cm);
\draw[thick] (0,0) circle (1 cm);

\draw[fill=white,thick] (-2,-.3) rectangle (-.5,.3);

\node at (-1,.5) {$\cdot$};
\node at (-1.15,.55) {$\cdot$};
\node at (-1.3,.6) {$\cdot$};

\node at (-1.3,0) {$g_2$};

\node[rotate = -30] at (-1.5,1) {\arrowlines};
\node[rotate = -30] at (-1.3,.9) {\arrowlines};
\node[rotate = -30] at (-.8,.6) {\arrowlines};
\end{tikzpicture}};

\node at (7,-1) {.};

\end{tikzpicture}
\end{center}
\end{lemma}

\begin{proof}
This is an easy diagrammatic argument which uses the fact that $g_1 = hg_2h^{-1}$ for some $h \in \Sy{n}$. Replacing $g_1$ by $hg_2h^{-1}$, we slide $h$ around the diagram to cancel it with $h^{-1}$.
\end{proof}




\subsection{Diagrams as bimodule homomorphisms} \label{section-bimodules}

In order to establish an isomorphism between $\Hcenter$ and $\ShiftSym{}$, we will make use of some representations of 
the monoidal category $\Heisencat$ constructed in \cite{Kho14}.

To describe these representations, we start by setting some notation for $(\MB{C}[\Sy{k_1}], \MB{C}[\Sy{k_2}])$-bimodules. All inclusions are assumed to be the standard ones $\symembedding{k}{n}: \Sy{k} \rightarrow \Sy{n}$ introduced in Section \ref{sect-sym-group}. Suppose that $k_1,k_2 \leq n$. We write:

\begin{itemize}
	\item $(n)$ for $\MB{C}[\Sy{n}]$ considered as a $(\MB{C}[\Sy{n}],\MB{C}[\Sy{n}])$-bimodule.
	\item $(n)_{k_2}$ for $\MB{C}[\Sy{n}]$ considered as a $(\MB{C}[\Sy{n}],\MB{C}[\Sy{k_2}])$-bimodule.
	\item $_{k_1}(n)$ for $\MB{C}[\Sy{n}]$ considered as a $(\MB{C}[\Sy{k_1}],\MB{C}[\Sy{n}])$-bimodule.
	\item $_{k_1}(n)_{k_2}$ for $\MB{C}[\Sy{n}]$ considered as a $(\MB{C}[\Sy{k_1}],\MB{C}[\Sy{k_2}])$-bimodule.
\end{itemize}

Let $\bimodcat$ be the category whose objects are compositions of induction and restriction functors of symmetric groups. We write
\begin{equation*}
\ind^{n+1}_n := \ind^{\Sy{n+1}}_{\Sy{n}} \quad\quad \text{and} \quad\quad \res^{n+1}_n := \res^{\Sy{n+1}}_{\Sy{n}}.
\end{equation*}
Since induction from $\Sy{n}$ to $\Sy{n+1}$ is given by tensoring on the left by $(n+1)_{n}$ and restriction from $\Sy{n+1}$ to $\Sy{n}$ is given by tensoring on the left by $_n(n+1)$, the objects in $\bimodcat$ can be reinterpreted as $(\MB{C}[\Sy{k_1}], \MB{C}[\Sy{k_2}])$-bimodules for $k_1,k_2 \geq 0$. 

\begin{example}
One object in $\bimodcat$ is the composition
\begin{equation} \label{eqn-functor-comp}
\res^{5}_{4} \circ \ind_{4}^{5} \circ \ind_{3}^{4} \circ \res_{3}^{4}.
\end{equation}
In the language of bimodules, this is the $(\MB{C}[\Sy{4}],\MB{C}[\Sy{4}])$-bimodule
\begin{equation*}
_{4}(5)_{4}(4)_{3}(4).
\end{equation*}
\end{example}

The morphisms in $\bimodcat$ are certain natural transformations of these compositions (or, equivalentely, certain bimodule homomorphisms). Like $\Heisencat$, morphisms in $\bimodcat$ can be presented diagrammatically as oriented compact 1-manifolds embedded in $\MB{R} \times [0,1]$. Unlike $\Heisencat$, in $\bimodcat$ we label the regions of the strip $\MB{R} \times [0,1]$ by non-negative integers, so that if there is an upwards oriented line separating two regions and the right region is labeled by $n$, then the left region must be labeled by $n+1$. The diagram
\begin{center}
		\begin{tikzpicture}
			\draw[->,thick] (1,0) to (1,1);
			\node at (1.5,0.5) {$n$};
			\node at (0,0.5) {$n+1$};
		\end{tikzpicture}
\end{center}
denotes the identity endomorphism of the induction functor $\ind_n^{n+1}$ or alternatively the identity endomorphism of the bimodule $(n+1)_n$. 

If there is a downward oriented line separating two regions and the right is labeled by $n+1$ then the left must be labeled by $n$. The diagram
\begin{center}
		\begin{tikzpicture}
			\draw[->,thick] (1,1) to (1,0);
			\node at (2,0.5) {$n+1$};
			\node at (0.5,0.5) {$n$};
		\end{tikzpicture}
\end{center}
denotes the identity endomorphism of the restriction functor $\res_{n}^{n+1}$ or alternatively the identity endomorphism of the bimodule $_n(n+1)$.

The bimodule maps associated to the four U-turns are:
    \begin{align} 
    	& \begin{tikzpicture}[baseline=(current bounding box.center)] 
    		\draw[->,thick] (0,0) arc (180:0:0.8); 
    		\node at (0.8,0.4) {$n$};
    		\node at (2.3,0.4) {$n+1$};
    	\end{tikzpicture}, \quad
    		(n+1)_n (n+1) \to (n+1), \quad g \otimes h \mapsto gh, \quad g,h \in S_{n+1}, \label{RCap} \\
    		\nonumber \\
    	& \begin{tikzpicture}[baseline=(current bounding box.center)]
    		\draw[->,thick] (0,0.7) arc(-180:0:0.8);
    		\node at (2.3,0.4) {$n$};
    		\node at (0.8,0.4) {$n+1$};
    	\end{tikzpicture}, \quad
    		(n) \to \ _n (n+1)_n, \quad g \mapsto g, \quad g \in S_n, \label{RCup} \\
    		\nonumber \\
    	& \begin{tikzpicture}[baseline=(current bounding box.center)]
    		\draw[<-,thick] (0,0) arc (180:0:0.8);
    		\node at (0.8,0.4) {$n+1$};
    		\node at (2.3,0.4) {$n$};
    	\end{tikzpicture}, \quad
    		_n (n+1)_n \to (n), \quad g \mapsto \pr{n}(g) = \begin{cases} g & g(n+1) = n+1 \\ 0 & \text{otherwise,}\\ \end{cases} \label{LCap} \\
    		\nonumber \\
    	& \begin{tikzpicture}[baseline=(current bounding box.center)] 
    		\draw[<-,thick] (0,0.7) arc(-180:0:0.8);
    		\node at (2.3,0.4) {$n+1$};
    		\node at (0.8,0.4) {$n$};
    	\end{tikzpicture}, \quad (n+1) \to (n+1)_n (n+1), \label{LCup} 
    \end{align}
where the last map is determined by the condition that
\begin{equation*}
1_{n+1} \mapsto \sum_{i = 1}^{n+1} s_i s_{i+1} \cdots s_n \otimes s_n \cdots s_{i+1} s_i = \sum_{g \in \speclcos{n+1}{n}} g \otimes g^{-1}. 
\end{equation*}
\newpage
Finally, the upward crossing is the bimodule map

\begin{equation} \label{UCross}
	\begin{tikzpicture}[baseline=(current bounding box.center)]
		\draw[->,thick] (0,0) to (1,1);
		\draw[->,thick] (1,0) to (0,1);
		\node at (1.3,0.5) {$n$};
		\node at (1.5,0.5) {};
	\end{tikzpicture}, \quad
		(n+2)_n \to (n+2)_n, \quad g \mapsto g s_{n+1}, \quad g \in S_{n+2}.
\end{equation}

Any diagram that has a region labeled with a negative number is set to $\zero$. It is shown in \cite{Kho14} that all diagrams are compatible with isotopy. 

\begin{remark} \label{remark-closed-diagrams-and-center}
Closed diagrams in $\bimodcat$ with outside region labeled by $n$ correspond to $(\MB{C}[\Sy{n}],\MB{C}[\Sy{n}])$-bimodule endomorphisms of $(n)$. The algebra of such bimodule endomorphisms is isomorphic to $Z(\MB{C}[\Sy{n}])$ via the map which sends $f \in \End_{(\MB{C}[\Sy{n}] , \MB{C}[\Sy{n}])}(\MB{C}[\Sy{n}])$ to $f(1_n)$. Thus closed diagrams in $\bimodcat$ may be regarded as elements of the center of the group algebra.
\end{remark}

\omitt{
\begin{proposition} \cite{Kho14} \label{prop-S'-satisfies-rln}
The diagrams in $\bimodcat$ satisfy relations the defining relations \eqref{up down double crossings}, \eqref{anti-clockwise and left curl}, \eqref{eqn-symmetric-group-relations} of $\Heisencat$.
\end{proposition}
}

Khovanov shows that the diagrams in $\bimodcat$ satisfy the defining relations for morphisms in $\Heisencat$.  As a result, given an endomorphism of $\Heisencat$, after labeling the far right region by a non-negative integer, one obtains a well-defined bimodule homomorphism in $\bimodcat$.  An additional relation that can be calculated directly from the definitions of oriented cups and caps is the following:

\begin{equation} \label{eqn-c_kn-ctilde_kn}
\begin{tikzpicture}

\draw[thick] (0,1) circle (20pt);
\node at (-.7,1) {\arrowlines};
\node at (0,1) {\scriptsize{$n$}};
\node at (1.2,1.5) {\scriptsize{$n$ + $1$}};
\node at (2.2,1) {$= \;\; n+1.$};

\end{tikzpicture}
\end{equation}
In other words, the endomorphism $c_0\in \Hcenter$ becomes the scalar $n+1$ in $Z(\MB{C}[\Sy{n+1}])$.
\omitt{
\draw[thick] (7,1) circle (20pt);
\node[rotate = 180] at (6.3,1) {\arrowlines};
\node at (7,1) {\scriptsize{$n$}};
\node at (8.2,1.5) {\scriptsize{$n$ - $1$}};
\node at (9.2,1) {$= 1$};
}

$\bimodcat$ is the direct sum of categories
\begin{equation*}
\bimodcat = \bigoplus_{k = 0}^\infty \bimodcatn{k},
\end{equation*}
where $\bimodcatn{k}$ contains all objects such that induction or restriction starts at $k$ (i.e. the rightmost region of the diagram is labeled by $k$). There are functors $\htobimod{k}: \Heisencat \rightarrow \bimodcatn{k}$ such that the object $\epsilon_1 \epsilon_2 \dots \epsilon_n$ is taken to a composition of induction and restriction functors with $+$ sent to $\ind_{i}^{i+1}$ and $-$ sent to $\res_{i-1}^{i}$ where $i$ in each case is determined by the requirement that induction/restriction begin from $\Sy{k}$. $\htobimod{k}$ takes a diagram from $\Heisencat$ to $\bimodcatn{k}$ by labeling regions so that the rightmost region is labeled with a $k$ and then interpreting the diagram as an element of $\bimodcatn{k}$.

\begin{example}
$\htobimod{5}: \Heisencat \rightarrow \bimodcatn{5}$ takes
\begin{equation*}
(++-+-) \quad \xmapsto{\;\;\; \htobimod{5} \;\;\;} \quad \ind_5^{6} \circ \ind_{4}^{5} \circ \res_{4}^{5} \circ \ind^{5}_4 \circ \res^{5}_4,
\end{equation*}
\begin{equation*}
(-++) \quad \xmapsto{\;\;\; \htobimod{5} \;\;\;} \quad \res^7_6 \ind^7_6 \ind_5^6.
\end{equation*}

\omitt{
and
\vspace{3mm}
\begin{center}
\begin{tikzpicture}

\draw[->,thick] (3.3,4) to [out=270,in = 90] (2.3,0);
\draw[->,thick] (6.3,4) to [out=270,in = 270] (1.8,4);
\draw[->,thick] (3.8,0) to [out =90, in = 270] (4.8,4);
\draw[->,thick] (5.3,0) to [out = 90, in =270] (.3,4);
\draw[black,->,thick] (5.3,2) circle (4mm);

\node at (4.9,2) {\arrowlines};

\node at (7.3,2) {$\xmapsto{\;\;\; \htobimod{5} \;\;\;}$};

\draw[->,thick] (13,4) to [out=270,in = 90] (12,0);
\draw[->,thick] (16,4) to [out=270,in = 270] (11.5,4);
\draw[->,thick] (13.5,0) to [out =90, in = 270] (14.5,4);
\draw[->,thick] (15,0) to [out = 90, in =270] (10,4);

\node at (15,2) {\scriptsize{5}};
\node at (15,3.3) {\scriptsize{4}};
\node at (14.2,.5) {\scriptsize{6}};
\node at (12.7,.7) {\scriptsize{7}};
\node at (13.4,2.2) {\scriptsize{6}};
\node at (13.7,3.2) {\scriptsize{5}};
\node at (11,1.2) {\scriptsize{6}};
\node at (12,2.6) {\scriptsize{5}};
\node at (12.3,3.5) {\scriptsize{4}};

\draw[very thick] (9.5,0) to [out = 105, in = 270] (9.2,2) to [out = 90,in = 255] (9.5,4);
\draw[very thick] (16.5,0) to [out = 75, in = 270] (16.7,2) to [out = 90,in = 285] (16.5,4);

\node at (8.8,2) {\Large{$5 \times $}};

\node at (16.8,0) {$.$};

\end{tikzpicture}
\end{center}

Notice that the clockwise oriented circle on the left is sent to the scalar 5 in accordance with \eqref{eqn-c_kn-ctilde_kn}.

}
\end{example}

\vspace{6mm}


In the remainder of this section we calculate the image of a number of important diagrams in $\Heisencat$ under the functors $\htobimod{k}$.
\begin{lemma} \label{lemma-right-curl-to-JM} \cite[Section 4]{Kho14}
The diagram 
\begin{center}
\begin{tikzpicture}

\draw[thick,rotate = 180] (4,1) .. controls (4,1.5) and (4.7,1.5) .. (4.9,1);
          	\draw[thick,rotate = 180] (4,1) .. controls (4,0.5) and (4.7,0.5) .. (4.9,1);
          	\draw[thick,rotate = 180,<-] (5,0) .. controls (5,0.5) .. (4.9,1) ;
          	\draw[thick,rotate = 180] (4.9,1) .. controls (5,1.5) .. (5,2);

\node at (-4,.-.2) {\scriptsize{$n$-$1$}};
\node at (-4.4,-.9) {\scriptsize{$n$-$2$}};
\node at (-6,-.7) {\scriptsize{$n$}};

\end{tikzpicture}
\end{center}
is the endomorphism of $(n)_{n-1}$ which is right multiplication by $\JM{n}$.
\end{lemma}

\begin{proof}
The right twist curl can be written as the composition of a cup, a crossing, and a cap.
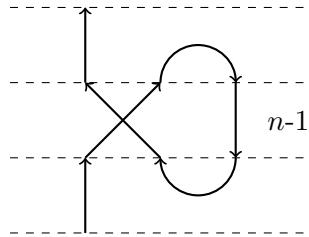
\vspace{2mm}
 \begin{center}
    	\begin{tikzpicture}
    		\draw[->,thick] (0,0) to (0,1);
    		\draw[->,thick] (0,1) to (1,2);
    		\draw[->,thick] (1,1) to (0,2);
    		\draw[<-,thick] (2,1) to (2,2);
    		\draw[->,thick] (2,1) arc(0:-180:0.5);
    		\draw[->,thick] (0,2) to (0,3);
    		\draw[->,thick] (1,2) arc(180:0:0.5);
    		\draw[dashed] (-1,0) to (3,0);
    		\draw[dashed] (-1,1) to (3,1);
    		\draw[dashed] (-1,2) to (3,2);
    		\draw[dashed] (-1,3) to (3,3);
    		\node at (2.7,1.5) {$n$-$1$};
    	\end{tikzpicture}
\end{center}
Applying the endomorphism to $1_{n}$ gives
    \begin{align*}
    	1_{n} &\mapsto \sum_{i = 1}^{n-1} s_i \cdots s_{n-2} \otimes s_{n-2} \cdots s_i \mapsto \sum_{i = 1}^{n-1} s_i \cdots s_{n-2} s_{n-1} \otimes s_{n-2} \cdots s_i \\
    	    &\mapsto \sum_{i = 1}^{n-1} s_i \cdots s_{n-2} s_{n-1} s_{n-2} \cdots s_i = \JM{n}
    \end{align*}
where the equality holds by \eqref{eqn-alternative-JM-pres}.
\end{proof}

\begin{lemma} \label{lemma-S-calculations}
Let $k \leq n$:
\begin{enumerate}
\item The diagram
\begin{center}
\begin{tikzpicture}

\draw[thick] (.9,-.5) arc (180:360:.6cm);
\draw[thick] (0,-.5) arc (180:360:1.5cm);
\draw[thick] (-.85,-.5) arc (180:360:2.4cm);

\node at (.9,-.5) {\arrowlines};
\node at (0,-.5) {\arrowlines};
\node at (-.85,-.5) {\arrowlines};

\node at (1.5,-.7) {\scriptsize{$n$-$k$}};
\node at (-.3,-1) {\scriptsize{$n$-$1$}};
\node at (-1.1,-1.2) {\scriptsize{$n$}};

\draw[thick,black] (.7,-.8) circle (.2mm);
\draw[thick,black] (.55,-.85) circle (.2mm);
\draw[thick,black] (.4,-.9) circle (.2mm);

\end{tikzpicture}
\end{center}
corresponds to the bimodule homomorphism $(n) \rightarrow (n)_{n-k}(n)$ which sends
\begin{equation*}
1_n \mapsto \sum_{g \in \speclcos{n}{n-k}} g \otimes g^{-1}.
\end{equation*}

\item Let $\mu \vdash k$ and $x_1, x_2 \in (n)$. The diagram
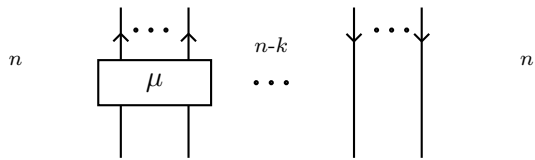
\begin{center}
\begin{tikzpicture}

\draw[thick] (2.3,-1)--(2.3,1);
\draw[thick] (1.4,-1)--(1.4,1);

\draw[fill=white,thick] (2.6,-.3) rectangle (1.1,.3);

\node at (2.3,.6) {\arrowlines};
\node at (1.4,.6) {\arrowlines};

\draw[thick,black] (2,.7) circle (.2mm);
\draw[thick,black] (1.8,.7) circle (.2mm);
\draw[thick,black] (1.6,.7) circle (.2mm);

\draw[thick,black] (4.8,.7) circle (.2mm);
\draw[thick,black] (5,.7) circle (.2mm);
\draw[thick,black] (5.2,.7) circle (.2mm);

\draw[thick,black] (3.2,0) circle (.2mm);
\draw[thick,black] (3.4,0) circle (.2mm);
\draw[thick,black] (3.6,0) circle (.2mm);

\draw[thick] (4.5,-1)--(4.5,1);
\draw[thick] (5.4,-1)--(5.4,1);

\node[rotate = 180] at (4.5,.6) {\arrowlines};
\node[rotate = 180] at (5.4,.6) {\arrowlines};


\node at (1.85,0) {$\mu$};









\node at (3.4,.5) {\scriptsize{$n$-$k$}};
\node at (0,.3) {\scriptsize{$n$}};
\node at (6.8,.3) {\scriptsize{$n$}};

\end{tikzpicture}
\end{center}
corresponds to the bimodule homomorphism $(n)_{n-k}(n) \rightarrow (n)_{n-k}(n)$ which sends
\begin{equation*}
x_1 \otimes x_2 \mapsto x_1\partitioncyclen{\mu}{n} \otimes x_2.
\end{equation*}
\end{enumerate}
\end{lemma}

\begin{proof}
These follow from direct calculation using the definitions of cups, caps, and crossings.
\end{proof}

\omitt{
For $\lambda \in \partitions$, we define
\begin{equation*}
\ckimage{k}{n} := \htobimod{n}(c_k), \quad\quad \cktildeimage{k}{n} := \htobimod{n}(\tilde{c}_k), \quad\quad \text{and} \quad\quad \alpha_{\lambda,n} := \htobimod{n}(\alpha_\lambda).
\end{equation*} 
Observe that by \eqref{eqn-c_kn-ctilde_kn}, $\ckimage{0}{n} = \alpha_{1,n} = n$. 
}

\begin{lemma} \label{lemma-what-is-c_k-in-center}
As elements of $Z(\MB{C}[\Sy{n}])$:
\begin{enumerate}
\item \label{eqn-ck-value} $ \displaystyle \htobimod{n}(c_k) = \sum_{i = 1}^n s_i \cdots s_{n-1} \JM{n}^k s_{n-1} \cdots s_i,$
\item \label{eqn-cktilde-value} $\htobimod{n}(\tilde{c}_k) = \pr{n}(\JM{n+1}^k).$
\item \label{eqn-alpha-value} $\htobimod{n}(\alpha_\mu) = \begin{cases}
\classsum{\mu}{n} & \text{if $|\mu| \leq n$}\\
0 & \text{otherwise}.
\end{cases}$
\end{enumerate}
\end{lemma}

\begin{proof}
\eqref{eqn-ck-value}-\eqref{eqn-cktilde-value} are found in \cite{Kho14} Section 4 and can be computed from the definitions of cups and caps and Lemma \ref{lemma-right-curl-to-JM}. \eqref{eqn-alpha-value} can be computed by composing the maps in Lemma \ref{lemma-S-calculations} with a sequence of $|\mu|$ nested clockwise oriented caps from \eqref{RCap}. When $|\mu| > n$ then $\htobimod{n}(\alpha_{\mu})$ will have its inner region labeled by $n-|\mu| < 0$ and will therefore be $0$.
\end{proof}

\section{The isomorphism $\primaryiso: \Hcenter \longrightarrow \ShiftSym{}$}
In this section we establish the algebra isomorphism $\Hcenter \cong \ShiftSym{}$. The proof is somewhat analogous to Ivanov and Kerov's proof of a related isomorphism connecting shifted symmetric functions to the representation theory of symmetric groups (see Theorem 9.1 in \cite{IK99}).


In \cite{Kho14} Section 4, Khovanov defines a grading on $\Hcenter$ by setting
\begin{equation} \label{eqn-disturb-def}
\deg(c_0) := 0, \quad \text{and} \quad \deg(c_k) = k+1, \quad \text{for $k \geq 1.$}
\end{equation}
We will consider the increasing filtration induced by this grading. A relationship between the elements $\{c_k\}_{k \geq 0}$ and $\{\alpha_k\}_{k \geq 1}$ is then given in terms of this filtration as follows.

\omitt{
and we denote the $k$th filtered part of $\Hcenter$ by $\disturb{k}(\Hcenter)$ so that
\begin{equation*}
\MB{C}[c_0] = \disturb{0}(\Hcenter) \subseteq \disturb{1}(\Hcenter) \subseteq \disturb{2}(\Hcenter) \subseteq \disturb{3}(\Hcenter) \dots
\end{equation*}

When $x \in \Hcenter$ we write 
\begin{equation*}
\disturbdeg(x) = k \quad \Longleftrightarrow \quad x \in \disturb{k}(\Hcenter) \;\;\; \text{but} \;\;\; x \notin \disturb{k-1}(\Hcenter).
\end{equation*}
} 

\begin{proposition} \label{lemma-highest-degree-term}
For any $k \geq 1$,
\begin{equation*}
\alpha_{k} = c_{k-1} + \loworderterms
\end{equation*}
\end{proposition}

\begin{proof}
This follows from repeated application of the dot sliding moves \eqref{k dots 1}-\eqref{k dots 2} and bubble sliding move \eqref{bubble clockwise}. Notice that with each application of these moves, we get a single term from the same filtered part plus additional terms of lower degree. 
\end{proof}

\omitt{

\begin{proof}
By Lemma \ref{lemma-alpha-value} and \eqref{eqn-what-is-c_n} for $n \geq k$, $\htobimod{n}(\alpha_{k}) = \classsum{k}{n}$ and $\htobimod{n}(c_{k-1}) = \classsum{k}{n} + \loworderterms$ so that
\begin{equation*}
\distsym(\htobimod{n}(\alpha_{k} - c_{k-1})) < k.
\end{equation*}
Proposition \ref{prop-F-and-degree} then implies that $\disturbdeg(\alpha_{k} - c_{k-1}) < k$. Since $\disturbdeg(\alpha_{k}) = \disturbdeg(c_{k-1}) = k$ the result follows.
\end{proof}

}
Since the elements $c_0, c_1, \dots$ are algebraically independent generators of $\Hcenter$, we immediately obtain the following.
\begin{corollary} \label{cor-end-generators}
The elements $\alpha_1, \alpha_2, \dots$ are algebraically independent generators of $\Hcenter$.
\end{corollary}

For any $\lambda \vdash n$, composing $\htobimod{n}$ with the normalized character $\norcharrep{\lambda}$ gives a map

\begin{equation*}
(\norcharrep{\lambda} \circ \htobimod{n}): \Hcenter \rightarrow \MB{C}
\end{equation*}
and allows us to define a homomorphism $\primaryiso: \Hcenter \rightarrow \funonyd$. Specifically, for $x \in \Hcenter$, we write
\begin{equation*}
[\primaryiso (x)](\lambda) := (\norcharrep{\lambda} \circ \htobimod{n})(x).
\end{equation*}

Combining Lemma \ref{lemma-what-is-c_k-in-center}.\ref{eqn-alpha-value} with \eqref{prop-normalized-char-class-sum} implies that for $\mu \vdash k$
\begin{equation} \label{eqn-value-alpha}
[\primaryiso(\alpha_\mu)](\lambda) = \begin{cases}
\frac{(n \downharpoonright k)}{\dim\simplerep{\lambda}}\chi^{\lambda}(\mu) & \text{if $k \leq n$}\\
0 & \text{otherwise.}
\end{cases}
\end{equation} 

\begin{theorem} \label{thm-main-1}
The map $\primaryiso$ induces an algebra isomorphism $\Hcenter \rightarrow \ShiftSym{} \subseteq \funonyd$ with
\begin{equation*}
\alpha_\mu \xmapsto{\primaryiso} \shiftpwr{\mu}.
\end{equation*}
\end{theorem}

\begin{proof}
Let $\lambda \vdash n$. $\primaryiso$ is an algebra homomorphism because $\htobimod{n}$ is a homomorphism from $\Hcenter$ to $Z(\MB{C}[\Sy{n}])$ and $\norcharrep{\lambda}$ is a homomorphism when restricted to $Z(\MB{C}[\Sy{n}])$. By Proposition \ref{prop-value-prwshift} and \eqref{eqn-value-alpha}, $\alpha_\mu$ maps to $\shiftpwr{\mu}$. Since the $\{\shiftpwr{k}\}_{k \geq 1}$ (respectively $\{\alpha_k\}_{k \geq 1}$) are algebraically independent generators of $\ShiftSym{}$ (resp. $\Hcenter$), $\primaryiso$ must be an isomorphism.
\end{proof} 

Note that Theorem \ref{thm-main-1} along with Lemma \ref{lemma-only-cycle-type-matters} imply that when $\mu \vdash n$, 
\begin{equation} \label{eqn-image-of-cong-class}
\begin{tikzpicture}

\node at (0,0) {\begin{tikzpicture}[scale = .8]

\draw[thick] (0,0) circle (1.8 cm);
\draw[thick] (0,0) circle (1.6 cm);
\draw[thick] (0,0) circle (1 cm);

\draw[fill=white,thick] (-2,-.3) rectangle (-.5,.3);

\node at (-1,.5) {$\cdot$};
\node at (-1.15,.55) {$\cdot$};
\node at (-1.3,.6) {$\cdot$};

\node at (-1.3,0) {\small{\text{$\congclasssum{\mu}{n}$}}};

\node[rotate = -30] at (-1.5,1) {\arrowlines};
\node[rotate = -30] at (-1.3,.9) {\arrowlines};
\node[rotate = -30] at (-.8,.6) {\arrowlines};
\end{tikzpicture}};


\node at (2.2,0) {$=$};

\node at (4.5,0) {\begin{tikzpicture}[scale = .8]

\node at (-2.5,0) {$\frac{n!}{z_{\mu,n}}$};

\draw[thick] (0,0) circle (1.8 cm);
\draw[thick] (0,0) circle (1.6 cm);
\draw[thick] (0,0) circle (1 cm);

\draw[fill=white,thick] (-2,-.3) rectangle (-.5,.3);

\node at (-1,.5) {$\cdot$};
\node at (-1.15,.55) {$\cdot$};
\node at (-1.3,.6) {$\cdot$};

\node at (-1.3,0) {\small{$\lambda$}};

\node[rotate = -30] at (-1.5,1) {\arrowlines};
\node[rotate = -30] at (-1.3,.9) {\arrowlines};
\node[rotate = -30] at (-.8,.6) {\arrowlines};
\end{tikzpicture}};

\node at (7.95,.3) {$\primaryiso$};
\draw[thick, ->] (7,0)--(8.9,0);
\draw[thick] (7,.1)--(7,-.1);

\node at (10,0) {$\frac{n!}{z_{\mu,n}}\shiftpwr{\mu}.$};

\end{tikzpicture}
\end{equation}

For $\lambda \vdash n$ recall that $\idempotent{\lambda}$ is the Young idempotent associated to $\lambda$. 

\begin{theorem}\label{thm:schur}
The isomorphism $\primaryiso$ sends 
\vspace{3mm}
\begin{center}
\begin{tikzpicture}

\node at (-3,-.3) {$\dim \simplerep{\lambda}$};
\node at (-3,.3) {$1$};
\draw[thick] (-3.8,0) -- (-2.3,0);

\node at (6,0) {$\shiftschur{\lambda}.$};

\draw[thick,->] (3,0) -- (4.5,0);
\draw[thick] (3,-.1) -- (3,.1);
\node at (3.75,.3) {$\primaryiso$};

\draw[thick] (0,0) circle (1.8 cm);
\draw[thick] (0,0) circle (1.6 cm);
\draw[thick] (0,0) circle (1 cm);

\draw[fill=white,thick] (-2,-.3) rectangle (-.5,.3);

\node at (-1,.5) {$\cdot$};
\node at (-1.15,.55) {$\cdot$};
\node at (-1.3,.6) {$\cdot$};

\node at (-1.3,0) {$\Youngidempotent{\lambda}$};

\node[rotate = -30] at (-1.5,1) {\arrowlines};
\node[rotate = -30] at (-1.3,.9) {\arrowlines};
\node[rotate = -30] at (-.8,.6) {\arrowlines};
\end{tikzpicture}
\end{center}
\end{theorem}

\begin{proof}
Recall that
\begin{equation*}
\Big(\frac{1}{\dim \simplerep{\lambda}}\Big)\idempotent{\lambda} = \sum_{\mu \vdash n} \frac{\charrep{\lambda}(\mu)}{n!} \congclasssum{\mu}{n},
\end{equation*}
while
\begin{equation*}
\shiftschur{\lambda} = \sum_{\mu \vdash n} \frac{\charrep{\lambda}(\mu)}{z_{\mu,n}}\shiftpwr{\mu}.
\end{equation*}
The result then follows from \eqref{eqn-image-of-cong-class}.
\end{proof}

The previous theorems gave graphical realizations of some important bases of $\ShiftSym{}$.  Now we go the other way, and describe Khovanov's curl generators 
$\ctilde{k}$ and $c_k$ as elements of $\ShiftSym{}$.  It is this description that makes an explicit connection between $\Heisencat$ and the transition and co-transition measures of Kerov.

\begin{theorem} \label{thm-moment-images}
The isomorphism $\primaryiso$ sends:
\begin{enumerate}
\item $\ctilde{k} \mapsto \moment{k} \in \ShiftSym{}$,
\item $c_k \mapsto \shiftpwr{1}\comoment{k} = \Boolean{k+2} \in \ShiftSym{}$.
\end{enumerate}
\end{theorem}

\begin{proof}
Let $\lambda \vdash n$, then from Lemma \ref{lemma-what-is-c_k-in-center} and Proposition \ref{prop-alt-definition-moments} we have
\begin{equation*}
[\primaryiso(\ctilde{k})](\lambda) = \norcharrep{\lambda}(\pr{n}(\JM{n+1}^k)) = \moment{k}(\lambda)
\end{equation*}
and
\begin{equation*}
[\primaryiso(c_{k})](\lambda) = \norcharrep{\lambda}\Big(\sum_{i = 1}^n s_i \cdots s_{n-1} \JM{n}^k s_{n-1} \cdots s_i \Big) = \shiftpwr{1}(\lambda)\comoment{k}(\lambda) = \Boolean{k+2}(\lambda).
\end{equation*}
\end{proof}

\begin{remark} \label{remark-FH}
In \cite{FH59}, Farahat and Higman used the inductive structure of symmetric groups to construct a $\MB{C}$-algebra known as the Farahat-Higman algebra $\FH_{\MB{C}}$ (see also Example 24, Section I.7, \cite{Mac15}). 
It follows from, for example \cite{IK99}, that there is an algebra isomorphism $\FH_{\MB{C}}\cong \ShiftSym{}$, and 
the functors $\htobimod{n}$ can also be used to give a direct isomorphism between $\Hcenter$ and $\FH_{\MB{C}}$. So in principle all of the appearances of shifted symmetric functions in the previous sections could be rephrased in the language of the Farahat-Higman algebra.
\end{remark}

\begin{remark} \label{remark-curl-recursive-relations}
Theorem \ref{thm-moment-images} and Remark \ref{remark-moments-as-sym} together imply that the recursive relationships for $\{\moment{k}\}$ and $\{\Boolean{k}\}$ in Remark \ref{rmk-recursive-for-boolean-moments} and $\{c_k\}$ and $\{\ctilde{k}\}$ in Lemma \ref{lemma-c-ctilde-rln} are both consequences of the well-known relationship between the elementary and homogeneous symmetric functions:
\begin{equation*}
\sum_{i = 0}^k (-1)^i e_ih_{n-i} = 0.
\end{equation*}
\end{remark}

\begin{example}
In $\ShiftSym{}$ we have $\shiftpwr{(2)}\shiftpwr{(2)} = \shiftpwr{(2,2)} + 4\shiftpwr{(3)} + 2\shiftpwr{(1,1)}$. In $\Hcenter$ the local relations can be used to compute the corresponding equation:
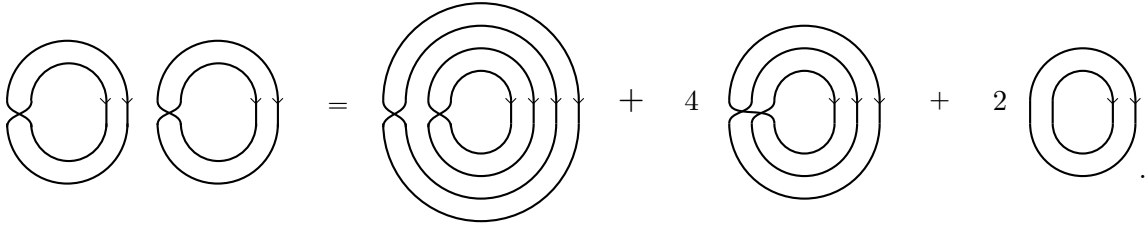
\begin{center}
\begin{tikzpicture}



\draw[thick] (0,0) arc (0:180:8mm);
\draw[thick] (-.28,0) arc (0:180:5mm);

\draw[thick] (2,0) arc (0:180:8mm);
\draw[thick] (1.71,0) arc (0:180:5mm);


\draw[thick] (-1.6,0) to [in = 90, out = 270] (-1.275,-.35);
\draw[thick] (-1.6,-.35) to [in = 270, out = 90] (-1.275,0);

\draw[thick] (0,0)--(0,-.35);
\draw[thick] (-.28,0)--(-.28,-.35);

\draw[thick] (.4,0) to [in = 90, out = 270] (.71,-.32);
\draw[thick] (.4,-.32) to [in = 270, out = 90] (.71,0);

\draw[thick] (2,0)--(2,-.35);
\draw[thick] (1.71,0)--(1.71,-.35);


\draw[thick] (0,-.3) arc (180:0:-8mm);
\draw[thick] (-.28,-.3) arc (180:0:-5mm);

\draw[thick] (2,-.3) arc (180:0:-8mm);
\draw[thick] (1.71,-.3) arc (180:0:-5mm);


\node at (2.8,-.1) {$=$};



\draw[thick] (6,0) arc (0:180:13mm);
\draw[thick] (5.7,0) arc (0:180:10mm);

\draw[thick] (5.4,0) arc (0:180:7mm);
\draw[thick] (5.1,0) arc (0:180:4mm);


\draw[thick] (3.4,0) to [in = 90, out = 270] (3.7,-.35);
\draw[thick] (3.4,-.35) to [in = 270, out = 90] (3.7,0);

\draw[thick] (4,0) to [in = 90, out = 270] (4.3,-.35);
\draw[thick] (4,-.35) to [in = 270, out = 90] (4.3,0);

\draw[thick] (6,0)--(6,-.3);
\draw[thick] (5.7,0)--(5.7,-.3);
\draw[thick] (5.4,0)--(5.4,-.3);
\draw[thick] (5.1,0)--(5.1,-.3);


\draw[thick] (6,-.3) arc (180:0:-13mm);
\draw[thick] (5.7,-.3) arc (180:0:-10mm);

\draw[thick] (5.4,-.3) arc (180:0:-7mm);
\draw[thick] (5.1,-.3) arc (180:0:-4mm);


\node at (6.7,0) {\Large{$+$}};


\node at (7.5,0) {4};


\draw[thick] (10,0) arc (0:180:10mm);
\draw[thick] (9.7,0) arc (0:180:7mm);
\draw[thick] (9.4,0) arc (0:180:4mm);


\draw[thick] (8,0) to [out = 270, in = 90] (8.6,-.3);
\draw[thick] (8.3,0) to [out = 270, in = 90] (8,-.3);
\draw[thick] (8.6,0) to [out = 270, in = 90] (8.3,-.3);

\draw[thick] (10,0) -- (10,-.3);
\draw[thick] (9.7,0) -- (9.7,-.3);
\draw[thick] (9.4,0) -- (9.4,-.3);


\draw[thick] (10,-.3) arc (180:0:-10mm);
\draw[thick] (9.7,-.3) arc (180:0:-7mm);
\draw[thick] (9.4,-.3) arc (180:0:-4mm);


\node at (10.8,0) {$+$};


\node at (11.6,0) {$2$};


\draw[thick] (13.4,0) arc (0:180:7mm);
\draw[thick] (13.1,0) arc (0:180:4mm);


\draw[thick] (13.4,0) -- (13.4,-.3);
\draw[thick] (13.1,0) -- (13.1,-.3);
\draw[thick] (12.3,0) -- (12.3,-.3);
\draw[thick] (12,0) -- (12,-.3);


\draw[thick] (13.4,-.3) arc (180:0:-7mm);
\draw[thick] (13.1,-.3) arc (180:0:-4mm);


\node[rotate = 90] at (-.29,0) {\arrow};
\node[rotate = 90] at (0,0) {\arrow};
\node[rotate = 90] at (1.71,0) {\arrow};
\node[rotate = 90] at (2,0) {\arrow};
\node[rotate = 90] at (6,0) {\arrow};
\node[rotate = 90] at (5.7,0) {\arrow};
\node[rotate = 90] at (5.4,0) {\arrow};
\node[rotate = 90] at (5.1,0) {\arrow};
\node[rotate = 90] at (10,0) {\arrow};
\node[rotate = 90] at (9.7,0) {\arrow};
\node[rotate = 90] at (9.4,0) {\arrow};
\node[rotate = 90] at (13.4,0) {\arrow};
\node[rotate = 90] at (13.1,0) {\arrow};

\node at (13.5,-1) {.};


\end{tikzpicture}
\end{center}
\end{example}


\subsection{Involutions on $\Hcenter$} 

In \cite{Kho14}, Khovanov introduced three involutive autoequivalences on $\Heisencat$. Only one of these, which we denote as $\involution$, acts non-trivially on $\Hcenter$ where it gives an involutive algebra automorphism. For $D \in \Hom_{\Heisencat}(Q_{\epsilon_1}, Q_{\epsilon_2})$, we have
\begin{equation*}
\involution(D) := (-1)^{\crossingplusdots(D)}D
\end{equation*}
where $\crossingplusdots(D)$ is the total number of dots and crossings in the diagram. Thus, in $\Hcenter$:
\begin{flalign}
& c_k \quad \xmapsto{\quad \involution \quad} \quad (-1)^k c_k, \\
&\ctilde{k} \quad \xmapsto{\quad \involution \quad} \quad (-1)^k\ctilde{k}, \\
& \alpha_{k} \quad \xmapsto{\quad \involution \quad} \quad (-1)^{k-1}\alpha_{k}  \label{eqn-involution-on-pwr}.
\end{flalign}
In Section 4 of \cite{OO97}, Okounkov and Olshanski identified an involutive algebra automorphism \\$\shiftinvolution:\ShiftSym{} \rightarrow \ShiftSym{}$ which acts on $f \in \ShiftSym{}$ such that for $\lambda \in \partitions$,
\begin{equation*}
[\shiftinvolution(f)](\lambda) = f(\lambda'),
\end{equation*}
where $\lambda'$ is the conjugate partition to $\lambda$. In particular
\begin{flalign}
& \shiftinvolution(\shiftschur{\lambda}) = \shiftschur{\lambda'}, \\
& \shiftinvolution(\elementaryshift{k}) = \homogenshift{k}, \\
&\shiftinvolution(\shiftpwr{k}) = (-1)^{k-1}\shiftpwr{k} \label{eqn-involution-on-power}.
\end{flalign}

\begin{proposition}
The involution $\involution$ on $\Hcenter$ coincides with the involution $\shiftinvolution$ on $\ShiftSym{}$.
\end{proposition}

\begin{proof}
This follows from the fact that $\{\alpha_k\}_{k \geq 1}$, (respectively $\{\shiftpwr{k}\}_{k \geq 1}$) generate $\Hcenter$ (resp. $\ShiftSym{}$), $\primaryiso(\alpha_k) = \shiftpwr{k}$, and a comparison of \eqref{eqn-involution-on-pwr} and \eqref{eqn-involution-on-power}. 
\end{proof}


\subsection{A graphical construction of the action of $\Walgebra$ on $\ShiftSym{}$}

In \cite{CLLS15}, the trace  $\mbox{Tr}(\Heisencat)$ (or zeroth Hochschild homology) of $\Heisencat$ is shown to be isomorphic as an algebra to a quotient of the W-algebra $\Walgebra$.  Like the center 
$\Hcenter$, which is the algebra of closed planar diagrams, the trace $\mbox{Tr}(\Heisencat)$ has a purely graphical description, as the space of annular diagrams modulo Khovanov's local diagrammatic relations.  More precisely, 
the underlying vector space of $\mbox{Tr}(\Heisencat)$ is isomorphic to the span of annular diagrams, where an annular diagram $\tilde{f}$ is by definition a diagram obtained by taking an endomorphism $f\in \mbox{End}_{\Heisencat}(X)$ for some object $X\in \Heisencat$, and closing it up to the right in an annulus. The multiplication in $\mbox{Tr}(\Heisencat)$ is given by gluing annuli around one another:

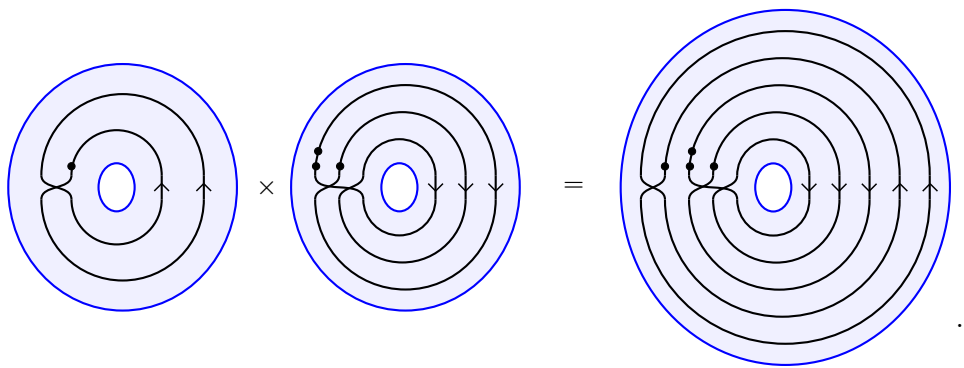
\begin{center}
\begin{tikzpicture}[scale = .8]

\node at (.3,0) {\begin{tikzpicture}[scale = .4]



\draw[thick,blue,fill = lightblue] (27,-.4) ellipse (38mm and 41mm);




\draw[thick] (29.7,0) arc (0:180:27mm);
\draw[thick] (28.3,0) arc (0:180:15mm);


\draw[thick] (24.3,0) to [out = 270, in = 90] (25.3,-.8);
\draw[thick] (25.3,0) to [out = 270, in = 90] (24.3,-.8);

\draw[thick] (29.7,0) -- (29.7,-.8);
\draw[thick] (28.3,0) -- (28.3,-.8);


\draw[thick] (29.7,-.8) arc (180:0:-27mm);
\draw[thick] (28.3,-.8) arc (180:0:-15mm);

\node[rotate = 270,scale = 1.4] at (29.7,-.4) {\arrow};
\node[rotate = 270,scale = 1.4]  at (28.3,-.4) {\arrow};


\draw[fill = black] (25.3,.3) circle (1.2mm);


\draw[thick,blue,fill = white] (26.8,-.4) ellipse (6mm and 8mm);

\end{tikzpicture}};

\node at (2.7,0) {$\times$};

\node at (5,0) {\begin{tikzpicture}[scale = .4]



\draw[thick,blue,fill = lightblue] (27,-.4) ellipse (38mm and 41mm);




\draw[thick] (30,0) arc (0:180:30mm);
\draw[thick] (29,0) arc (0:180:21mm);
\draw[thick] (28,0) arc (0:180:12mm);


\draw[thick] (24,0) to [out = 270, in = 90] (25.6,-.8);
\draw[thick] (24.8,0) to [out = 270, in = 90] (24,-.8);
\draw[thick] (25.6,0) to [out = 270, in = 90] (24.8,-.8);

\draw[thick] (30,0) -- (30,-.8);
\draw[thick] (29,0) -- (29,-.8);
\draw[thick] (28,0) -- (28,-.8);


\draw[thick] (30,-.8) arc (180:0:-30mm);
\draw[thick] (29,-.8) arc (180:0:-21mm);
\draw[thick] (28,-.8) arc (180:0:-12mm);

\node[rotate = 90,scale = 1.4] at (30,-.4) {\arrow};
\node[rotate = 90,scale = 1.4]  at (29,-.4) {\arrow};
\node[rotate = 90,scale = 1.4]  at (28,-.4) {\arrow};


\draw[fill = black] (24.03,.3) circle (1.2mm);
\draw[fill = black] (24.1,.8) circle (1.2mm);
\draw[fill = black] (24.83,.3) circle (1.2mm);


\draw[thick,blue,fill = white] (26.8,-.4) ellipse (6mm and 8mm);

\end{tikzpicture}};

\node at (7.8,0) {$=$};

\node at (11.5,0) {\begin{tikzpicture}[scale = .4]



\draw[thick,blue,fill = lightblue] (27.2,-.4) ellipse (54.72mm and 59.04mm);




\draw[thick] (32,0) arc (0:180:48mm);
\draw[thick] (31,0) arc (0:180:39mm);
\draw[thick] (30,0) arc (0:180:30mm);
\draw[thick] (29,0) arc (0:180:21mm);
\draw[thick] (28,0) arc (0:180:12mm);


\draw[thick] (22.4,0) to [out = 270, in = 90] (23.2,-.8);
\draw[thick] (23.2,0) to [out = 270, in = 90] (22.4,-.8);
\draw[thick] (24,0) to [out = 270, in = 90] (25.6,-.8);
\draw[thick] (24.8,0) to [out = 270, in = 90] (24,-.8);
\draw[thick] (25.6,0) to [out = 270, in = 90] (24.8,-.8);

\draw[thick] (32,0) -- (32,-.8);
\draw[thick] (31,0) -- (31,-.8);
\draw[thick] (30,0) -- (30,-.8);
\draw[thick] (29,0) -- (29,-.8);
\draw[thick] (28,0) -- (28,-.8);


\draw[thick] (32,-.8) arc (180:0:-48mm);
\draw[thick] (31,-.8) arc (180:0:-39mm);
\draw[thick] (30,-.8) arc (180:0:-30mm);
\draw[thick] (29,-.8) arc (180:0:-21mm);
\draw[thick] (28,-.8) arc (180:0:-12mm);

\node[rotate = 270,scale = 1.4] at (32,-.4) {\arrow};
\node[rotate = 270,scale = 1.4]  at (31,-.4) {\arrow};
\node[rotate = 90,scale = 1.4] at (30,-.4) {\arrow};
\node[rotate = 90,scale = 1.4]  at (29,-.4) {\arrow};
\node[rotate = 90,scale = 1.4]  at (28,-.4) {\arrow};


\draw[fill = black] (23.2,.3) circle (1.2mm);
\draw[fill = black] (24.03,.3) circle (1.2mm);
\draw[fill = black] (24.1,.8) circle (1.2mm);
\draw[fill = black] (24.83,.3) circle (1.2mm);


\draw[thick,blue,fill = white] (26.8,-.4) ellipse (6mm and 8mm);

\node at (33,-5) {.};

\end{tikzpicture}};

\end{tikzpicture}
\end{center}

The action of $\mbox{Tr}(\Heisencat)$ on $\Hcenter$ then acquires a graphical description: given an annular diagram $\tilde{f}\in \mbox{Tr}(\Heisencat)$ and a closed planar diagram $f\in \Hcenter$, the closed planar diagram $\tilde{f}g\in \Hcenter$ is given by inserting a planar neighborhood of the closed diagram $g$ into the middle of the annulus:

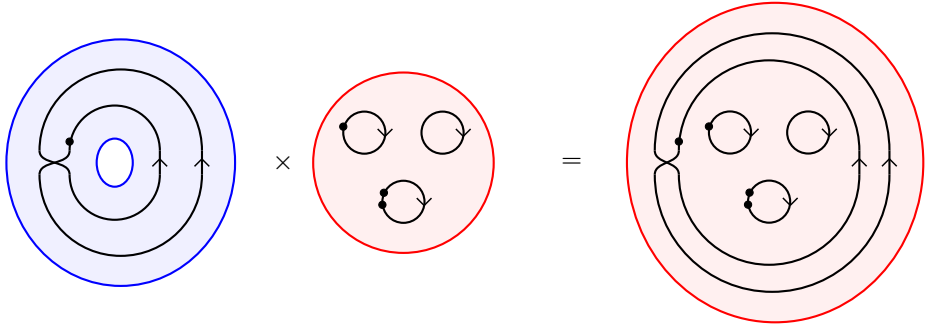
\begin{center}
\begin{tikzpicture}[scale = .8]

\node at (.3,0) {\begin{tikzpicture}[scale = .4]



\draw[thick,blue,fill = lightblue] (27,-.4) ellipse (38mm and 41mm);




\draw[thick] (29.7,0) arc (0:180:27mm);
\draw[thick] (28.3,0) arc (0:180:15mm);


\draw[thick] (24.3,0) to [out = 270, in = 90] (25.3,-.8);
\draw[thick] (25.3,0) to [out = 270, in = 90] (24.3,-.8);

\draw[thick] (29.7,0) -- (29.7,-.8);
\draw[thick] (28.3,0) -- (28.3,-.8);


\draw[thick] (29.7,-.8) arc (180:0:-27mm);
\draw[thick] (28.3,-.8) arc (180:0:-15mm);

\node[rotate = 270,scale = 1.4] at (29.7,-.4) {\arrow};
\node[rotate = 270,scale = 1.4]  at (28.3,-.4) {\arrow};


\draw[fill = black] (25.3,.3) circle (1.2mm);


\draw[thick,blue,fill = white] (26.8,-.4) ellipse (6mm and 8mm);

\end{tikzpicture}};

\node at (3,0) {$\times$};


\node at (5,0) {\begin{tikzpicture}[scale = .4]



\draw[thick,red,fill = lightred] (27,-.2) circle (30mm);




\draw[thick] (25.7,.8) circle (7mm);
\draw[thick] (28.3,.8) circle (7mm);
\draw[thick] (27,-1.5) circle (7mm);






\node[rotate = 90,scale = 1.4] at (27.7,-1.5) {\arrow};
\node[rotate = 90,scale = 1.4]  at (29,.8) {\arrow};
\node[rotate = 90,scale = 1.4]  at (26.4,.8) {\arrow};


\draw[fill = black] (25,1) circle (1.2mm);
\draw[fill = black] (26.35,-1.2) circle (1.2mm);
\draw[fill = black] (26.3,-1.6) circle (1.2mm);



\end{tikzpicture}};

\node at (7.8,0) {$=$};

\node at (11.5,0) {\begin{tikzpicture}[scale = .4]



\draw[thick,red,fill = lightred] (27.2,-.4) ellipse (49.25mm and 53.136mm);




\draw[thick] (31,0) arc (0:180:39mm);
\draw[thick] (30,0) arc (0:180:30mm);


\draw[thick] (23.2,0) to [out = 270, in = 90] (24,-.8);
\draw[thick] (24,0) to [out = 270, in = 90] (23.2,-.8);

\draw[thick] (31,0) -- (31,-.8);
\draw[thick] (30,0) -- (30,-.8);


\draw[thick] (31,-.8) arc (180:0:-39mm);
\draw[thick] (30,-.8) arc (180:0:-30mm);

\node[rotate = 270,scale = 1.4] at (30,-.4) {\arrow};
\node[rotate = 270,scale = 1.4]  at (31,-.4) {\arrow};


\draw[fill = black] (24,.3) circle (1.2mm);

\draw[thick] (25.7,.6) circle (7mm);
\draw[thick] (28.3,.6) circle (7mm);
\draw[thick] (27,-1.7) circle (7mm);

\node[rotate = 90,scale = 1.4] at (27.7,-1.7) {\arrow};
\node[rotate = 90,scale = 1.4]  at (29,.6) {\arrow};
\node[rotate = 90,scale = 1.4]  at (26.4,.6) {\arrow};

\draw[fill = black] (25,.8) circle (1.2mm);
\draw[fill = black] (26.35,-1.4) circle (1.2mm);
\draw[fill = black] (26.3,-1.8) circle (1.2mm);



\node at (33,-5) {.};

\end{tikzpicture}};

\end{tikzpicture}
\end{center}

Thus, via the isomorphisms
\[
	\Hcenter \cong \ShiftSym{},  \ \ \ \ \mbox{Tr}(\Heisencat) \cong \Walgebra
\]
of Theorem \ref{thm-main-1} and \cite{CLLS15}, respectively, we obtain a purely graphical construction of the action of $\Walgebra$ on $\ShiftSym{}$.  Such an action was first considered by Lascoux-Thibon in \cite{LT01}.

\bibliographystyle{amsplain}
\bibliography{HeisenbergRefs}

\end{document}